\documentclass[reqno]{amsart}

\usepackage{amsmath}
\usepackage{amsthm}
\usepackage{amssymb}
\usepackage{graphicx}

\theoremstyle{plain}
\newtheorem{theorem}{Theorem}[section]
\newtheorem{lemma}[theorem]{Lemma}
\newtheorem{proposition}[theorem]{Proposition}
\newtheorem{corollary}[theorem]{Corollary}

\theoremstyle{remark}
\newtheorem{remark}{Remark}
\newtheorem{exmp}[remark]{Example}

\numberwithin{equation}{section}

\newcommand{\C}{\mathbb{C}}
\newcommand{\D}{\mathbb{D}}
\newcommand{\N}{\mathbb{N}}

\newcommand{\T}{\mathbb{T}}
\newcommand{\Z}{\mathbb{Z}}
\newcommand{\abs}[1]{\lvert#1\rvert}
\newcommand{\nor}[1]{\lVert#1\rVert}

\newcommand{\mcB}{\mathcal{B}}
\newcommand{\mcF}{\mathcal{F}}

\newcommand{\mrmT}{\mathrm{T}}

\DeclareMathOperator{\cspn}{\overline{\mathrm{sp}}}
\DeclareMathOperator{\Tr}{Tr}

\title[\null]{Baxter's inequality for finite predictor coefficients of multivariate
long-memory stationary processes}

\author[\null]{AKIHIKO INOUE, YUKIO KASAHARA and MOHSEN POURAHMADI}

\address{A.\ Inoue\\
Department of Mathematics\\
Hiroshima University\\
Higashi-Hiroshima 739-8526\\
Japan}
\email{inoue100@hiroshima-u.ac.jp}

\address{Y.\ Kasahara\\
Department of Mathematics \\
Hokkaido University \\
Sapporo 060-0810 \\
Japan}
\email{y-kasa@math.sci.hokudai.ac.jp}

\address{M.\ Pourahmadi\\
Department of Statistics\\
Texas A\&M University\\
College Station, TX 77843\\
USA}
\email{pourahm@stat.tamu.edu}

\begin{document}

\subjclass[2010]{Primary 60G25; secondary 62M20, 62M10.}

\keywords{Baxter's inequality, predictor coefficients, multivariate stationary processes,
long memory, partial autocorrelation functions, phase functions.}


\begin{abstract}
For a multivariate stationary process, we develop explicit representations
for the finite predictor coefficient matrices, the finite
prediction error covariance matrices and the partial autocorrelation function (PACF)
in terms of the Fourier coefficients of its phase function in the spectral domain.
The derivation is based on a novel alternating projection technique and the use of
the forward and backward innovations corresponding to  predictions based on the
infinite past and future, respectively.
We show that such representations are ideal for studying the rates of convergence of the finite
predictor coefficients, prediction error covariances, and the PACF as well as for proving a multivariate
version of Baxter's inequality for a multivariate FARIMA process
with a common fractional differencing order for all components of the process.
\end{abstract}

\maketitle

\section{Introduction}\label{sec:1}

Baxter's inequality in \cite{Bax} provides valuable information about the  convergence of
the finite predictor coefficients to their infinite past counterparts (autoregressive coefficients)
of a short-memory univariate stationary process.
It has been used by \cite{Ber} in proving the consistency of the autoregressive
model fitting process and the corresponding autoregressive spectral density estimator,
and in proving the validity of autoregressive sieve bootstrap
for a stationary time series in \cite{Bu1, Bu2, KPP}. Due to the widespread applicability
of Baxter's inequality in these areas and others, there has been a great deal of activities
in extending it to the setups of multivariate stationary processes in
\cite{HD,CP},
random fields in \cite{MJK}, and rectangular arrays in \cite{MMP}.
In these extensions, the \textit{boundedness}\/ of the spectral density function of
the underlying process appears to be an absolutely essential and indispensable
part of proving Baxter's inequality.

In \cite{IK2}, however, Baxter's inequality was established for univariate
long-memory processes where the boundedness of the spectral density function is clearly violated.
Unlike the classical proofs for short-memory processes involving the orthogonal polynomials or
the Durbin--Levinson algorithm,
the key ingredient of the proof in \cite{IK2} was an
explicit representation of the finite predictor coefficients in terms of
the autoregressive (AR) and moving average (MA) coefficients.
The derivation of the representation in turn was based on techniques that use von Neumann's
alternating projections on the infinite past and future. These techniques were first used by
\cite{I1} and have been developed 
to derive the needed representations for
the finite prediction error variances (\cite{I1, I2, IK1, I3}),
the partial autocorrelation
functions (\cite{I3, BIK, KB}),
and the finite predictor coefficients (\cite{IK2}).
Unfortunately, most of the details of the proofs in the univariate case
do not carry over to the multivariate setup where, for example, all functions and the
sequences of AR and MA coefficients are matrix-valued and hence in general
do not commute with each other.

In this paper, for a multivariate stationary process, we prove the desired explicit representations
for the finite predictor coefficients, the finite prediction error covariances and the
partial autocorrelation function (PACF). 
See Theorems \ref{thm:pac947}--\ref{thm:fpc548} in Section \ref{sec:5}. 
The three new ingredients that enable us to obtain the results in the multivariate framework are:
\begin{itemize}
\item[(i)] Use of the Fourier coefficients of the matrix-valued phase function
of the process in the spectral domain, rather than the AR and MA coefficient matrices 
(see Section \ref{sec:4}).
\item[(ii)] Development of an enhanced alternating projection technique tailored to
the specific needs of the problem at hand (see Section \ref{sec:3}). 
\item[(iii)] Use of the forward and backward innovation processes corresponding to
the predictions based on the infinite past and future, respectively 
(see Sections \ref{sec:2}, \ref{sec:4} and \ref{sec:5}).
\end{itemize}

Our representation theorems make it possible to extend Baxter's inequality and other univariate
asymptotic results to the multivariate long-memory processes.
Even when specialized to univariate processes,
our method and results are more succinct, transparent and improve the known univariate results in
several ways. For example, our representation theorem for the finite predictor coefficients,
i.e., Theorem \ref{thm:fpc548} below, is stated under the minimality condition
(see (M) in Section \ref{sec:5}) only,
which is weaker than the condition
in the corresponding univariate result, i.e., Theorem 2.9 in \cite{IK2}.

In this paper, when applying the representation theorems,
we restrict our attention
to a class of $q$-variate long-memory processes,
that is, the $q$-variate \textit{FARIMA}\/ (fractional autoregressive integrated moving-average) or
\textit{vector ARFIMA}\/
processes with common fractional differencing order for all components. A process $\{X_k\}$ in this class
has the spectral density $w$ of the form
\begin{equation}
w(e^{i\theta})=|1-e^{i\theta}|^{-2d}g(e^{i\theta})g(e^{i\theta})^*,
\label{eq:farima529}
\end{equation}
where $d\in (-1/2,1/2)\setminus \{0\}$ and $g:\T\to \C^{q\times q}$ has rational entries
satisfying some suitable conditions; see (F) in Section \ref{sec:6}.
The process $\{X_k\}$ is described by the equation
\begin{equation}
(1-L)^d X_k=g(L)\xi_k, \qquad k\in\Z,
\label{eq:farima-eq723}
\end{equation}
where $L$ is the lag operator defined by $LX_{m}=X_{m-1}$ and
$\{\xi_k\}$ is a $q$-variate white noise, i.e., a $q$-variate, centered process
such that $E[\xi_n\xi_m^*]=\delta_{nm}I_q$ with $I_q$ being the $q\times q$ unit matrix.
See, e.g., \cite{Chung}. We notice that the parameter $d$ in (\ref{eq:farima529}) is the
{\it fractional differencing degree}\/ in (\ref{eq:farima-eq723}).
The $q$-variate FARIMA processes are multivariate analogues of univariate ones
introduced independently by \cite{GJ} and \cite{Ho}.

We present the following quick
summary of the asymptotic results
obtained by applying our representation theorems to a $q$-variate FARIMA process $\{X_k\}$
with (\ref{eq:farima529}):
\begin{enumerate}
\item[(1)]Baxter's inequality for $\{X_k\}$ with $d\in (0,1/2)$
(see Theorem \ref{thm:Baxter02} below).
\item[(2)]The precise asymptotics for the finite prediction error covariances
$v_n$ and $\tilde{v}_n$ of $\{X_k\}$ with $d\in(-1/2,1/2) \setminus \{0\}$
(see Theorem \ref{thm:asymp752} below; see also Section \ref{sec:5} for
the definitions of $v_n$ and $\tilde{v}_n$).
\item[(3)]The precise asymptotic behavior for the PACF $\alpha_n$ of $\{X_k\}$ with
$d\in(-1/2,1/2)\setminus\{0\}$
(see Theorem \ref{thm:PACF431} below; see also Section \ref{sec:5} for
the definition of $\alpha_n$).
\end{enumerate}

First, Baxter's inequality for FARIMA processes  is of the form
\begin{equation}
\sum_{j=1}^{n}\nor{\phi_{n,j} - \phi_j}
\le K\sum_{j=n+1}^{\infty}\Vert \phi_j\Vert,\qquad
n\in\N,
\label{eq:Baxter597}
\end{equation}
for some positive constant $K$,
where, for $a\in\C^{q\times q}$, $\Vert a\Vert$ denotes the
spectral norm of $a$ (see Section \ref{sec:2}),
and $\phi_j$ and $\phi_{n,j}$
denote the forward
infinite and finite predictor coefficients, respectively, of $\{X_k\}$ (see Sections
\ref{sec:2} and \ref{sec:5}, respectively, for their precise definitions).
We also prove a backward analogue of (\ref{eq:Baxter597});
see Corollary \ref{cor:backward-baxter519} below.
We refer to \cite{IK2} for the corresponding result for univariate
long-memory processes and \cite{BK, PGM, RS} for its application; 
see also \cite{ICM} for other applications of results in \cite{IK2}. 
In \cite{CP}, Baxter's inequality (\ref{eq:Baxter597}) was proved
for a class of multivariate short-memory stationary processes.
The original inequality (\ref{eq:Baxter597}) of
Baxter \cite{Bax} was an assertion for univariate short-memory processes.
See also \cite{Ber} and \cite[Section 7.6.2]{Pou}.

Next, the asymptotic results in (2) above are of the form
\begin{align}
v_n = v_{\infty} + \frac{d^2}{n}v_{\infty} + O(n^{-2}),
\qquad
n\to\infty,
\label{eq:v-asymp842}\\
\tilde{v}_n = \tilde{v}_{\infty} + \frac{d^2}{n}\tilde{v}_{\infty} + O(n^{-2}),
\qquad
n\to\infty,
\label{eq:v-tilde-asymp842}
\end{align}
where $v_{\infty}$ (resp., $\tilde{v}_{\infty}$) is the forward (resp., backward) infinite
prediction error covariance of $\{X_k\}$; see Section \ref{subsec:6.3} for
their precise definitions.
We refer to \cite{I1, I2, IK1, I3} for
the corresponding results for univariate
long-memory processes. See also \cite{IS, Gi}
for related work.

Finally, the result in (3) is of  the form
\begin{equation}
\alpha_n = \frac{d}{n}V + O(n^{-2}),
\qquad
n\to\infty,
\label{eq:PACF741}
\end{equation}
where $V$ is a unitary matrix in $\C^{q\times q}$ which depends only on $g$ (and not $d$).
We refer to \cite{I1, I2, IK1, I3, BIK}
for the corresponding results for univariate long-memory processes.
In the theory of orthogonal polynomials on the unit circle, the PACF appears as the
sequence of \textit{Verblunsky coefficients} and plays a central role.
See, e.g., \cite{DPS, Bin, KB}.

The above $q$-variate FARIMA process has a common fractional differencing order $d$ for all components.
The question arises of proving analogues of (1)--(3) above for more general $q$-variate FARIMA processes
which have, in general, different order of
differencing in each component, i.e.,
\[
(1-L)^{\mathbf{d}}:=
\left(
\begin{matrix}
(1 - L)^{d_1}   &        & 0             \cr
                & \ddots &               \cr
0               &        & (1 - L)^{d_q}
\end{matrix}
\right)
\]
with $\mathbf{d}=(d_1,\ldots,d_q)$, instead of $(1-L)^d$
(see, e.g., \cite{Chung}).
We leave this question open here; the difficulty stems from the fact that,
for such a general $q$-variate FARIMA process,
the matrices $g(L)$ and $(1-L)^{\mathbf{d}}$
do not commute with each other.

This paper is organized as follows.
In Section \ref{sec:2}, we give preliminary definitions and basic facts.
In Section \ref{sec:3}, we prove the key projection theorem.
In Section \ref{sec:4}, we describe some basic facts about the Fourier coefficients of
the phase function which is needed in Section \ref{sec:5}.
In Section \ref{sec:5}, we prove the main results, i.e.,
the representation theorems for the finite prediction error covariances,
the PACF and the finite predictor coefficients of multivariate stationary processes.
In Section \ref{sec:6}, we apply the main results to multivariate FARIMA processes
with common fractional differencing order for all components,
and establish the results (1)--(3) above for them.


\section{Preliminaries}\label{sec:2}

Let $\C^{m\times n}$ be the set of all complex $m\times n$ matrices;
we write $\C^q$ for $\C^{q\times 1}$. 
We write $I_n$ for the $n\times n$ unit matrix.
For $a\in \C^{m\times n}$, $a^{\mrmT}$ denotes the transpose of $a$, and
$\bar{a}$ and
$a^*$ the complex and Hermitian conjugates of $a$, respectively;
thus, in particular, $a^*:=\bar{a}^{\mrmT}$.
For $a\in\C^{q\times q}$, we write $\Vert a\Vert$ for the spectral norm of $a$:
\[
\Vert a\Vert:=\sup_{u\in \C^q, \vert u\vert = 1}\vert au\vert.
\]
Here $\vert u\vert:=(\sum_{i=1}^q\vert u^i\vert^2)^{1/2}$ denotes the Euclidean norm of
$u=(u^1,\dots,u^q)^{\mrmT}\in \C^q$. A Hermitian matrix
$a\in\C^{q\times q}$ is said to be \textit{positive}, denoted as $a\ge 0$,
if $(au)^*u\ge 0$ for all $u\in \C^q$.
When $a\ge 0$, we have
$\Vert a\Vert=\sup_{u\in \C^q, \vert u\vert = 1} (au)^*u$.
For Hermitian matrices $a, b\in\C^{q\times q}$, we
write $a\ge b$ if $a - b\ge 0$. If $a\ge b$, then we have $\Vert a\Vert \ge \Vert b\Vert$.
For $p\in [1,\infty)$ and $K\subset \Z$, $\ell_p^{q\times q}(K)$ denotes the space of
$\C^{q\times q}$-valued sequences
$\{a_k\}_{k\in K}$ such that $\sum_{k\in K}\Vert a_k\Vert^p<\infty$.
We write $\ell_{p+}^{q\times q}$ for $\ell_p^{q\times q}(\N\cup\{0\})$
and $\ell_{p+}$ for $\ell_{p+}^{1\times 1}=\ell_p^{1\times 1}(\N\cup\{0\})$.

Let $\T:=\{z\in\C :\vert z\vert=1\}$ be the unit circle in $\C$.
We write $\sigma$ for the normalized Lebesgue measure $d\theta/(2\pi)$ on
$([-\pi,\pi), \mcB([-\pi,\pi)))$, where $\mcB([-\pi,\pi))$ is the Borel $\sigma$-algebra
of $[-\pi,\pi)$; thus we have $\sigma([-\pi,\pi))=1$.
For $p\in [1,\infty)$, we write $L_p(\T)$ for the Lebesgue space of measurable functions $f:\T\to\C$
such that $\Vert f\Vert_p<\infty$, where
$\Vert f\Vert_p:=\{\int_{-\pi}^{\pi}\vert f(e^{i\theta})\vert^p\sigma(d\theta)\}^{1/p}$.
Let $L_p^{m\times n}(\T)$ be the space of $\C^{m\times n}$-valued functions on
$\T$ whose entries belong to $L_p(\T)$.

The Hardy class $H_2(\T)$ on $\T$ is the closed subspace of
$L_2(\T)$ consisting of $f\in L_2(\T)$ such that
$\int_{-\pi}^{\pi}e^{im\theta}f(e^{i\theta})\sigma(d\theta)=0$ for $m=1,2,\dots$.
Let $H_2^{m\times n}(\T)$ be the space of $\C^{m\times n}$-valued functions on
$\T$ whose entries belong to $H_2(\T)$.
Let $\D:=\{z\in\C : \vert z\vert<1\}$ be the open unit disk in $\C$.
We write $H_2(\D)$ for the Hardy class on $\D$, consisting of
holomorphic functions $f$ on $\D$ such that
$\sup_{r\in [0,1)}\int_{-\pi}^{\pi}\vert f(re^{i\theta})\vert^2\sigma(d\theta)<\infty$.
As usual, we identify each function $f$ in $H_2(\D)$ with its boundary function
$f(e^{i\theta}):=\lim_{r\uparrow 1}f(re^{i\theta})$, $\sigma$-a.e.,
in $H_2(\T)$.
A function $h$ in $H_2^{n\times n}(\T)$ is called \textit{outer}\/ if $\det h$ is a
$\C$-valued outer function, that is, $\det h$ satisfies
$\log\vert \det h(0)\vert
=\int_{-\pi}^{\pi}\log\vert \det h(e^{i\theta})\vert \sigma(d\theta)$
(cf.\ \cite[Definition 3.1]{KK}).

For $q\in\N$,
let $\{X_k\}=\{X_k:k\in\Z\}$ be a $\C^q$-valued, centered,
weakly stationary process, defined on a probability space $(\Omega, \mcF, P)$, which
we shall simply call a \textit{$q$-variate stationary process}. 
Write $X_k=(X^1_k,\dots,X^q_k)^{\mrmT}$, and
let $M$ be the complex Hilbert space
spanned by all the entries $\{X^j_k: k\in\Z,\ j=1,\dots,q\}$ in $L^2(\Omega, \mcF, P)$,
which has inner product $(x, y)_{M}:=E[x\overline{y}]$ and
norm $\Vert x\Vert_{M}:=(x,x)_{M}^{1/2}$.
For $K\subset \Z$ such as $\{n\}$,
$(-\infty,n]:=\{n,n-1,\dots\}$, $[n,\infty):=\{n,n+1,\dots\}$,
and $[m,n]:=\{m,\dots,n\}$ with $m\le n$,
we define the closed subspace $M_K^X$ of $M$ by
\[
M_K^X:=\cspn \{X^j_k: j=1,\dots,q,\ k\in K\}.
\]
We write $(M_K^X)^{\bot}$ for the orthogonal complement of $M_K^X$ in $M$.
Let $P_K$ and $P_K^{\perp}$ be the orthogonal projection operators of $M$ onto
$M_K^X$ and $(M_K^X)^{\perp}$, respectively.

Let $M^q$ be the space of $\C^q$-valued random variables on
$(\Omega, \mcF, P)$ whose entries belong to $M$.
The norm $\Vert x\Vert_{M^q}$ of $x=(x^1,\dots,x^q)^{\mrmT}\in M^q$ is given by
$\Vert x\Vert_{M^q}:=(\sum_{i=1}^q \Vert x^i\Vert_M^2)^{1/2}$.
For $K\subset\mathbb{Z}$ and
$x=(x^1,\dots,x^q)^{\mrmT}\in M^q$, we write $P_Kx$ for $(P_Kx^1, \dots, P_Kx^q)^{\mrmT}$.
We define $P_K^{\perp}x$ in a similar way.
For $x=(x^1,\dots,x^q)^{\mrmT}$ and $y=(y^1,\dots,y^q)^{\mrmT}$ in $M^q$,
\[
\langle x,y\rangle:=E[xy^*]=
\left(
\begin{matrix}
(x^1, y^1)_M & (x^1, y^2)_M & \cdots & (x^1, y^q)_M\cr
(x^2, y^1)_M & (x^2, y^2)_M & \cdots & (x^2, y^q)_M\cr
\vdots       & \vdots       & \ddots & \vdots      \cr
(x^q, y^1)_M & (x^q, y^2)_M & \cdots & (x^q, y^q)_M
\end{matrix}
\right)
\in \C^{q\times q}
\]
stands for the Gram matrix of $x$ and $y$.

Let $\{X_k\}$ be a $q$-variate stationary process.
If there exists a positive $q\times q$ Hermitian matrix-valued function $w$ on $\T$, satisfying
$w\in L^{q\times q}_1(\T)$ and
\[
\langle X_m, X_n\rangle = \int_{-\pi}^{\pi}e^{-i(m-n)\theta}w(e^{i\theta})\frac{d\theta}{2\pi},
\qquad n,m\in\Z,
\]
then we call $w$ the \textit{spectral density}\/ of $\{X_k\}$.
We say that $\{X_k\}$ is \textit{purely nondeterministic}\/ (PND) if
$\cap_{n\in\Z}M_{(-\infty,n]}^X=\{0\}$.
Every PND process $\{X_k\}$ has spectral density (cf.\ Section 4 in \cite[Chapter II]{Roz}).
We consider the following condition:
\[
\mbox{$\{X_k\}$ has spectral density $w$ such that $\log \det w\in L_1(\T)$.}
\tag{\rm A}
\]
A necessary and sufficient condition for (A) is that $\{X_k\}$ is PND and
its spectral density $w$ satisfies $\det w(e^{i\theta})>0$, $\sigma$-a.e. (see Theorem 6.1 in
\cite[Chapter II]{Roz}).

In what follows, we assume (A) for $\{X_k\}$.
Let $\{\tilde{X}_k: k\in\Z\}$ be the time-reversed process of $\{X_k\}$:
\begin{equation}
\tilde{X}_k:=X_{-k},\qquad k\in\Z.
\label{eq:t-reversed146}
\end{equation}
Then, since
\[
\langle \tilde{X}_n, \tilde{X}_m\rangle = \langle X_{-n}, X_{-m}\rangle
= \int_{-\pi}^{\pi}e^{-i(n-m)\theta}w(e^{-i\theta})\frac{d\theta}{2\pi},
\]
$\{\tilde{X}_k\}$ has the spectral density $\tilde{w}$ given by
\begin{equation}
\tilde{w}(e^{i\theta})=w(e^{-i\theta}).
\label{eq:spec628}
\end{equation}
In particular, $\{\tilde{X}_k\}$ also satisfies (A).
The spectral densities $w$ and $\tilde{w}$ have the
decompositions
\begin{equation}
w(e^{i\theta}) = h(e^{i\theta}) h(e^{i\theta})^*, \qquad
\tilde{w}(e^{i\theta})=\tilde{h}(e^{i\theta})\tilde{h}(e^{i\theta})^*,
\qquad \mbox{$\sigma$-a.e.},
\label{eq:decomp888}
\end{equation}
respectively,
for some outer functions $h$ and $\tilde{h}$ in $H_2^{q\times q}(\T)$, and
$h$ and $\tilde{h}$ are unique up to constant unitary factors
(see, e.g., \cite[Chapter II]{Roz} and \cite[Theorem 11]{HL}).
We define the outer function $h_{\sharp}$ in $H_2^{q\times q}(\T)$ by
\begin{equation}
h_{\sharp}(z) := \{\tilde{h}(\overline{z})\}^*.
\label{eq:outft628}
\end{equation}
Then, $h_{\sharp}$ satisfies 
\begin{equation}
w(e^{i\theta}) =h_{\sharp}(e^{i\theta})^*h_{\sharp}(e^{i\theta}),\quad
\mbox{$\sigma$-a.e.}
\label{eq:decomp111}
\end{equation}
We may take $h_{\sharp}=h$ for the univariate case $q=1$ but there is no such simple
relation between $h$ and $h_{\sharp}$ for $q\ge 2$. 
We call $h^*h_{\sharp}^{-1}$ the \textit{phase function} 
of $\{X_k\}$. 
Since
\[
\{h(e^{i\theta})^*h_{\sharp}(e^{i\theta})^{-1}\}^*h(e^{i\theta})^*h_{\sharp}(e^{i\theta})^{-1}
=\{h_{\sharp}(e^{i\theta})^*\}^{-1}w(e^{i\theta})h_{\sharp}(e^{i\theta})^{-1}=I_q
\]
holds $\sigma$-a.e., 
it is a unitary matrix valued function on $\T$. 
See Section \ref{sec:4} and \cite[page 428]{Pe}.

Let
\[
X_k=\int_{-\pi}^{\pi}e^{-ik\theta}\Lambda(d\theta),\qquad k\in\Z,
\]
be the spectral representation of $\{X_k\}$, where $\Lambda$ is the $\C^q$-valued
random spectral measure such that
\[
\left(
\int_{-\pi}^{\pi} \phi(e^{i\theta})   \Lambda(d\theta),
\int_{-\pi}^{\pi} \psi(e^{i\theta})   \Lambda(d\theta)
\right)_M
= \int_{^-\pi}^{\pi} \phi(e^{i\theta}) w(e^{i\theta}) \psi(e^{i\theta})^* \frac{d\theta}{2\pi}
\]
for $\phi, \psi\in L(w)$
with $L(w)$ being the class of measurable $\phi: \T \to \C^{1\times q}$ satisfying
$\int_{^-\pi}^{\pi} \phi(e^{i\theta}) w(e^{i\theta}) \phi(e^{i\theta})^* \sigma(d\theta)<\infty$
(cf.\ \cite[Chapter I]{Roz}).
We define a $q$-variate stationary process $\{\xi_k:k\in\Z\}$, called the \textit{forward innovation
process}\/ of $\{X_k\}$, by
\begin{equation}
\xi_k:=\int_{-\pi}^{\pi}e^{-ik\theta}h(e^{i\theta})^{-1}\Lambda(d\theta),\qquad k\in\Z.
\label{eq:xisp628}
\end{equation}
Then, $\{\xi_k\}$ satisfies $\langle \xi_n, \xi_m\rangle = \delta_{n m}I_q$ and
\begin{equation}
M_{(-\infty,n]}^X=M_{(-\infty,n]}^{\xi},\qquad n\in\Z
\label{eq:causal111}
\end{equation}
(cf.\ Section 4 in \cite[Chapter II]{Roz}), whence, for $n\in\Z$,
$\{\xi_k^j: j=1,\dots,q,\ k\ge n+1\}$ becomes a complete orthonormal basis of
$(M_{(-\infty,n]}^X)^{\bot}$.

On the other hand,
the spectral representation of $\{\tilde{X}_k\}$ is given by
\[
\tilde{X}_k=\int_{-\pi}^{\pi}e^{-ik\theta}\tilde{\Lambda}(d\theta),\qquad k\in\Z
\]
with the $\C^q$-valued random measure $\tilde{\Lambda}$ defined by
\begin{equation}
\tilde{\Lambda}(E):=\Lambda(-E),\qquad E\in \mcB((-\pi,\pi)),
\label{eq:randomsp628}
\end{equation}
where $-E:=\{-\theta: \theta\in E\}$. Let $\{\tilde{\xi}_k: k\in\Z\}$ be
the forward innovation process
of $\{\tilde{X}_k\}$ given by
\begin{equation}
\tilde{\xi}_k
:=\int_{-\pi}^{\pi}e^{-ik\theta}\tilde{h}(e^{i\theta})^{-1}
\tilde{\Lambda}(d\theta),\qquad k\in\Z.
\label{eq:etasp628}
\end{equation}
Then, we easily see that $\{\tilde{\xi}_k\}$ satisfies
$\langle \tilde{\xi}_n, \tilde{\xi}_m\rangle=\delta_{n m}I_q$ and
\begin{equation}
M_{[-n,\infty)}^X=M_{(-\infty, n]}^{\tilde{\xi}},\qquad n\in\Z,
\label{eq:causal222}
\end{equation}
whence, for $n\in\Z$,
$\{\tilde{\xi}_k^j: j=1,\dots,q,\ k\ge n+1\}$ becomes a complete orthonormal basis of
$(M_{[-n,\infty)}^X)^{\bot}$.
We also call $\{\tilde{\xi}_k\}$ the \textit{backward innovation process}\/ of $\{X_k\}$.
Then, $\{\xi_k\}$ turns out to be the backward innovation process of $\{\tilde{X}_k\}$.

We define, respectively, the \textit{forward MA and AR coefficients}\/ $c_k$ and $a_k$
of $\{X_k\}$ by
\begin{equation}
h(z)=\sum_{k=0}^{\infty}z^kc_k,\qquad
-h(z)^{-1}=\sum_{k=0}^{\infty}z^ka_k,\qquad z\in\D,
\label{eq:MAAR111}
\end{equation}
and the \textit{backward MA and AR coefficients}\/ $\tilde{c}_k$
and $\tilde{a}_k$ of $\{X_k\}$ by
\begin{equation}
\tilde{h}(z)=\sum_{k=0}^{\infty}z^k\tilde{c}_k,\qquad
-\tilde{h}(z)^{-1}=\sum_{k=0}^{\infty}z^k\tilde{a}_k,\qquad z\in\D.
\label{eq:MAAR222}
\end{equation}
It should be noticed that $c_k$ and $a_k$ (resp., $\tilde{c}_k$ and $\tilde{a}_k$)
are the backward (resp., forward) MA and AR coefficients of
the time-reversed process $\{\tilde{X}_k\}$, respectively.
All of $\{c_k\}$, $\{a_k\}$, $\{\tilde{c}_k\}$ and $\{\tilde{a}_k\}$ are $\C^{q\times q}$-valued
sequences,
and we have $\{c_k\}, \{\tilde{c}_k\}\in \ell_{2+}^{q\times q}$ and
$c_0a_0=\tilde{c}_0\tilde{a}_0=-I_q$.
We have the following forward and backward MA representations of $\{X_k\}$, respectively:
\begin{equation}
X_n=\sum_{k=-\infty}^nc_{n-k}\xi_k,\qquad
X_{-n}=\sum_{k=-\infty}^{n}\tilde{c}_{n-k}\tilde{\xi}_k,\qquad n\in\Z
\label{eq:MA428}
\end{equation}
(cf.\ Section 4 in \cite[Chapter II]{Roz}). If we further assume
\begin{equation}
\{a_k\}, \{\tilde{a}_k\}\in \ell_{1+}^{q\times q},
\label{eq:integrable432}
\end{equation}
then the following forward and backward AR representations of $\{X_k\}$, respectively,
also hold:
\begin{equation}
\sum_{k=-\infty}^na_{n-k}X_k + \xi_n = 0,\qquad
\sum_{k=-\infty}^{n}\tilde{a}_{n-k} X_{-k} + \tilde{\xi}_n = 0,\qquad n\in\Z
\label{eq:AR222}
\end{equation}
(see, e.g., the proof of \cite[Theorem 4.4]{I1}). From (\ref{eq:AR222}),
we obtain the following forward and backward infinite prediction formulas, respectively, for
$\{X_k\}$:
\[
P_{(-\infty, -1]}X_0 = \sum_{k=1}^{\infty}\phi_kX_{-k},\qquad
P_{[1, \infty)}X_0 = \sum_{k=1}^{\infty}\tilde{\phi}_kX_{k}.
\]
Here
\begin{equation}
\phi_k:=c_0a_k,\qquad \tilde{\phi}_{k}:=\tilde{c}_0\tilde{a}_k,\qquad k\in\N.
\label{eq:ipc574}
\end{equation}
We call $\phi_k$ (resp., $\tilde{\phi}_k$) the \textit{forward}\/ (resp., \textit{backward})
\textit{infinite predictor coefficients}\/ of $\{X_k\}$.
It should be noticed that $\phi_k$ (resp., $\tilde{\phi}_k$) are the backward
(resp., forward) infinite predictor coefficients of $\{\tilde{X}_k\}$.


\section{A Projection theorem}\label{sec:3}

In this section, we present a projection theorem which facilitates finding explicit
representations of the finite predictor coefficients, the finite prediction
error covariances and the PACF of a $q$-variate stationary process $\{X_k\}$, in terms of the
Fourier coefficients of the phase function.

Let $H$ be a Hilbert space with inner product $(\cdot, \cdot)$. Let $I:H\to H$ be the identity map.
For a closed subspace $A$ of $H$, we write $P_A$ for the orthogonal
projection operator of $H$ onto $A$ and $P_A^\perp$ for
that onto the orthogonal complement $A^{\perp}$ of $A$,
that is, $P_A^\perp=I - P_A$. For closed subspaces $A$ and $B$ of $H$,
von Neumann's Alternating Projection Theorem (cf.\ \cite[\S 9.6.3]{Pou}) states that
$(P_AP_B)^n$ converges to
$P_{A\cap B}$ as $n\to\infty$ in the strong operator topology.
From this, we have the following projection theorem.

\begin{theorem}[\cite{I1, I3}]\label{thm:poj01}
Let $A$ and $B$ be closed subspaces of $H$.
Then, we have, for $x,y\in H$,
\begin{align}
&P_{A\cap B}^\perp x= \sum_{k=0}^{\infty}
\left\{P_B^{\perp}(P_AP_B)^kx + P_A^{\perp}P_B(P_AP_B)^kx\right\},
\label{eq:PL111}\\
&
\begin{aligned}
&(P_{A\cap B}^\perp x, P_{A\cap B}^\perp y) = \sum_{k=0}^{\infty}
\left\{(P_B^{\perp}(P_AP_B)^kx, P_B^{\perp}(P_AP_B)^ky)\right. \\
&\qquad\qquad\qquad\qquad\quad
\left. + (P_A^{\perp}P_B(P_AP_B)^kx, P_A^{\perp}P_B(P_AP_B)^ky)\right\},
\end{aligned}
\label{eq:PL222}
\end{align}
the sum in $(\ref{eq:PL111})$ converging strongly.
\end{theorem}

The assertion (\ref{eq:PL222}) (resp., (\ref{eq:PL111})) is an abstract form of
\cite[Theorem 4.1]{I1} and \cite[Theorem 3.1]{I3} (resp., Remarks to \cite[Theorem 3.1]{I3}), and
can be proved in a similar way.

For our applications in this paper, we need
the next variant.

\begin{theorem}\label{thm:proj02}
Let $A$ and $B$ be closed subspaces of $H$.
Then, we have
\begin{align}
&P_{A\cap B}^\perp a= \sum_{k=0}^{\infty}
\left\{P_B^{\perp}(P_A^{\perp}P_B^{\perp})^ka - (P_A^{\perp}P_B^{\perp})^{k+1}a\right\},
\quad a\in A,
\label{eq:PL333}\\
&(P_{A\cap B}^\perp a_1, P_{A\cap B}^\perp a_2)
= \sum_{k=0}^{\infty}(P_B^{\perp}(P_A^{\perp}P_B^{\perp})^ka_1, a_2),
\quad a_1, a_2\in A,
\label{eq:PL444}\\
&(P_{A\cap B}^\perp a, P_{A\cap B}^\perp b)= -\sum_{k=0}^{\infty}((P_A^{\perp}P_B^{\perp})^{k+1}a, b),
\quad a\in A,\ b\in B,
\label{eq:PL555}
\end{align}
the sum in $(\ref{eq:PL333})$ converging strongly.
\end{theorem}

\begin{proof}
If $a\in A$, then
\[
\begin{aligned}
P_B^{\perp}P_AP_Ba
&= P_B^{\perp}(I - P_A^{\perp})P_Ba = -P_B^{\perp}P_A^{\perp}P_Ba
= -P_B^{\perp}P_A^{\perp}(I - P_B^{\perp})a\\
&=P_B^{\perp}P_A^{\perp}P_B^{\perp}a.
\end{aligned}
\]
Hence, we have, for $k=1,2,\dots$,
\[
P_B^{\perp}(P_AP_B)^ka = P_B^{\perp}(P_A^{\perp}P_B^{\perp})(P_AP_B)^{k-1}a
= \cdots = P_B^{\perp}(P_A^{\perp}P_B^{\perp})^{k}a,
\]
and, for $k=0,1,\dots$,
\[
\begin{aligned}
P_A^{\perp}P_B(P_AP_B)^ka
&=P_A^{\perp}(I - P_B^{\perp})(P_AP_B)^ka = -P_A^{\perp}P_B^{\perp}(P_AP_B)^ka\\
&= - (P_A^{\perp}P_B^{\perp})^{k+1}a.
\end{aligned}
\]
Therefore, (\ref{eq:PL333}) and
\begin{equation}
\begin{aligned}
&(P_{A\cap B}^\perp a_1, P_{A\cap B}^\perp a_2) = \sum_{m=0}^{\infty}
\left\{(P_B^{\perp}(P_A^{\perp}P_B^{\perp})^m a_1, P_B^{\perp}(P_A^{\perp}P_B^{\perp})^m a_2)\right. \\
&\qquad\qquad\qquad\quad
\left. + ((P_A^{\perp}P_B^{\perp})^{m+1} a_1, (P_A^{\perp}P_B^{\perp})^{m+1} a_2)\right\},
\quad a_1, a_2\in A
\end{aligned}
\label{eq:PL444b}
\end{equation}
follow from (\ref{eq:PL111}) and (\ref{eq:PL222}), respectively.
However, we have, for $a_1, a_2\in A$ and $m=0,1,\dots$,
\begin{align*}
&(P_B^{\perp}(P_A^{\perp}P_B^{\perp})^ma_1, P_B^{\perp}(P_A^{\perp}P_B^{\perp})^ma_2)
= (P_B^{\perp}(P_A^{\perp}P_B^{\perp})^{2m}a_1, a_2),\\
&((P_A^{\perp}P_B^{\perp})^{m+1}a_1, (P_A^{\perp}P_B^{\perp})^{m+1}a_2)
= (P_B^{\perp}(P_A^{\perp}P_B^{\perp})^{2m+1}a_1, a_2).
\end{align*}
Thus, (\ref{eq:PL444}) follows from (\ref{eq:PL444b}).

Let $a\in A$ and $b\in B$. Then, $(P_B^{\perp}a, P_B^{\perp}b)=0$.
For $m=1,2,\dots$, we have
\[
P_B^{\perp}(P_AP_B)^mb = P_B^{\perp}P_A(P_BP_A)^{m-1}b = - (P_B^{\perp}P_A^{\perp})^mb,
\]
whence
\[
\begin{aligned}
(P_B^{\perp}(P_AP_B)^ma, P_B^{\perp}(P_AP_B)^mb)
&= - (P_B^{\perp}(P_A^{\perp}P_B^{\perp})^{m}a, (P_B^{\perp}P_A^{\perp})^mb)\\
&= - ((P_A^{\perp}P_B^{\perp})^{2m}a, b).
\end{aligned}
\]
Similarly, we have, for $m=0,1,\dots$,
\[
P_A^{\perp}P_B(P_AP_B)^mb = P_A^{\perp}(P_BP_A)^mb = P_A^{\perp}(P_B^{\perp}P_A^{\perp})^mb,
\]
whence
\[
\begin{aligned}
(P_A^{\perp}P_B(P_AP_B)^ma, P_A^{\perp}P_B(P_AP_B)^mb)
&= - ((P_A^{\perp}P_B^{\perp})^{m+1}a, P_A^{\perp}(P_B^{\perp}P_A^{\perp})^mb)\\
&= - ((P_A^{\perp}P_B^{\perp})^{2m+1}a, b).
\end{aligned}
\]
Thus, (\ref{eq:PL555}) follows from (\ref{eq:PL222}).
\end{proof}

In the applications of this paper, $A$ and $B$ correspond to the infinite past and
future of a multivariate stationary process.


\section{Fourier coefficients of the phase function}\label{sec:4}

Let $\{X_k\}$ be a $q$-variate stationary process satisfying the condition (A),
with spectral density $w$. 
Let $\{\tilde{X}_k\}$, $h$ and $h_{\sharp}$ be as in Section \ref{sec:2}. 
We define a sequence $\{\beta_k\}_{k=-\infty}^{\infty}$ 
as the (minus of the) Fourier coefficients of the phase function 
$h^*h_{\sharp}^{-1}$:
\begin{equation}
\beta_k :=
- \int_{-\pi}^{\pi}e^{-ik\theta}h(e^{i\theta})^*h_{\sharp}(e^{i\theta})^{-1}\frac{d\theta}{2\pi},
\qquad k\in\Z.
\label{eq:beta-def667}
\end{equation}
Since $h^*h_{\sharp}^{-1}$ is unitary matrix valued (see Section \ref{sec:2}), we see that 
$\{\beta_k\}\in \ell_2^{q\times q}(\Z)$. 
The sequence $\{\beta_k\}$ plays a central role in our representation theorems.

Recall the forward and backward innovation processes $\{\xi_k\}$ and
$\{\tilde{\xi}_k\}$, respectively, of $\{X_k\}$ from Section \ref{sec:2}.

\begin{lemma}\label{lem:beta-xi897}
We assume {\rm (A)}. Then we have
\[
\langle\xi_{j}, \tilde{\xi}_{k}\rangle = -\beta_{j+k},\qquad
\langle\tilde{\xi}_{k}, \xi_{j}\rangle = -\beta_{k+j}^*,
\qquad j, k \in \Z.
\]
\end{lemma}

\begin{proof}
From (\ref{eq:outft628}), (\ref{eq:randomsp628}) and (\ref{eq:etasp628}), we see that
\[
\tilde{\xi}_k:=\int_{\pi}^{\pi}e^{ik\theta}\{h_{\sharp}(e^{i\theta})^*\}^{-1}
\Lambda(d\theta),\qquad k\in\Z.
\]
Combining this with (\ref{eq:decomp888}) and (\ref{eq:xisp628}), we obtain
\[
\begin{aligned}
\langle\xi_{j}, \tilde{\xi}_{k}\rangle
&=\int_{-\pi}^{\pi}e^{-i(j+k)\theta}
h(e^{i\theta})^{-1}h(e^{i\theta})h(e^{i\theta})^*h_{\sharp}(e^{i\theta})^{-1}\frac{d\theta}{2\pi}\\
&=\int_{-\pi}^{\pi}e^{-i(j+k)\theta}h(e^{i\theta})^*h_{\sharp}(e^{i\theta})^{-1}\frac{d\theta}{2\pi}
= -\beta_{j+k},
\end{aligned}
\]
which also implies the second equality.
\end{proof}

\begin{remark}
By Lemma \ref{lem:beta-xi897}, we have the following mutual representations between
$\{\xi_k\}$ and $\{\tilde{\xi}_k\}$:
\[
\xi_j
=-\sum_{k=-\infty}^\infty\beta_{j+k} \tilde{\xi}_{jk},
\qquad
\tilde{\xi}_k
=-\sum_{j=-\infty}^{\infty} \beta_{k+j}^*\xi_{j}.
\]
\end{remark}

\begin{lemma}\label{lem:beta-proper648}
We assume {\rm (A)}. Then,
for $\{s_l\}\in \ell_{2+}^{q\times q}$ and $n\in\Z$, we have
\begin{align}
&P_{[-n,\infty)}^{\perp}
\left(\sum_{l=0}^{\infty} s_l \xi_l\right)
=-\sum_{j=0}^\infty
\left(\sum_{l=0}^\infty s_l \beta_{n+j+l+1}\right)
\tilde{\xi}_{n+j+1},
\label{eq:proj555}\\
&P_{(-\infty,-1]}^{\perp}
\left(\sum_{l=0}^\infty s_l \tilde{\xi}_{n+l+1}\right)
=-\sum_{j=0}^\infty
\left(\sum_{l=0}^\infty s_l \beta_{n+j+l+1}^*\right)
\xi_j.
\label{eq:proj666}
\end{align}
In particular,
$\{\sum_{l=0}^\infty s_l \beta_{n+j+l+1}\}_{j=0}^{\infty},
\{\sum_{l=0}^\infty s_l \beta_{n+j+l+1}^*\}_{j=0}^{\infty} \in
\ell_{2+}^{q\times q}$.
\end{lemma}

\begin{proof}
By Lemma \ref{lem:beta-xi897}, we have
$\langle \sum_{l=0}^{\infty} s_l \xi_l, \tilde{\xi}_{n+j+1}\rangle
= - \sum_{l=0}^\infty s_l \beta_{n+j+l+1}$. On the other hand,
$\{\tilde{\xi}^k_{n+j+1}: k=1,\dots,q,\ j\ge 0\}$ is a complete orthonormal basis of
$(M_{[-n,\infty)}^X)^{\bot}$. Thus (\ref{eq:proj555}) follows.
We can prove (\ref{eq:proj666}) in a similar way.
\end{proof}

\begin{remark}
In Lemma \ref{lem:beta-proper648}, the map
$\{s_l\}_{l=0}^{\infty} \mapsto \{\sum_{l=0}^\infty s_l \beta_{n+j+l+1}\}_{j=0}^{\infty}$ defines a
bounded Hankel operator $\Gamma_n: \ell_{2+}^{q\times q} \to \ell_{2+}^{q\times q}$
with block Hankel matrix
\[
\left(
\begin{matrix}
\beta_{n+1} & \beta_{n+2} & \beta_{n+3} & \cdots \cr
\beta_{n+2} & \beta_{n+3} & \beta_{n+4} & \cdots \cr
\beta_{n+3} & \beta_{n+4} & \beta_{n+5} & \cdots \cr
\vdots      & \vdots      & \vdots      & \ddots
\end{matrix}
\right)
\]
(cf.\ \cite{Pe}), and similarly for $\{s_l\}_{l=0}^{\infty} \mapsto
\{\sum_{l=0}^\infty s_l \beta_{n+j+l+1}^*\}_{j=0}^{\infty}$.
\end{remark}

Lemma \ref{lem:beta-proper648} allows one to define,
for $n\in\mathbb{N}$ and $k\in\mathbb{N}\cup\{0\}$,
the sequences $\{b_{n,j}^k\}_{j=0}^{\infty}\in \ell_{2+}^{q\times q}$
by the recursion
\begin{equation}
b_{n,j}^0=\delta_{0j}I_q,
\ \ \
b_{n,j}^{2k+1}
=\sum_{l=0}^{\infty} b_{n,l}^{2k} \beta_{n+j+l+1},
\ \ \
b_{n,j}^{2k+2}
=\sum_{l=0}^{\infty} b_{n,l}^{2k+1} \beta_{n+j+l+1}^*.
\label{recursion}
\end{equation}
For $n\in\mathbb{N}$, we define the sequence $\{W_n^k\}_{k=0}^{\infty}$ in
$M^q$ by
\begin{align}
&W_n^{2k}
=P_{(-\infty,-1]}^{\perp}
(P_{[-n,\infty)}^\perp P_{(-\infty,-1]}^{\perp})^k X_0,\qquad k=0,1,\dots,
\label{eq:W-even216}\\
&W_n^{2k+1}
=-
(P_{[-n,\infty)}^{\perp} P_{(-\infty,-1]}^{\perp})^{k+1} X_0,\qquad k=0,1,\dots.
\label{eq:W-odd}
\end{align}

\begin{proposition}\label{prop:forward638}
We assume {\rm (A)}. Then, for $n\in\mathbb{N}$ and $k\in\mathbb{N}\cup\{0\}$, we have
\begin{equation}
W_n^{2k}
=c_0\sum_{j=0}^\infty b_{n,j}^{2k} \xi_j,
\qquad
W_n^{2k+1}
=c_0\sum_{j=0}^\infty b_{n,j}^{2k+1} \tilde{\xi}_{n+j+1}
\label{eq:expan111}
\end{equation}
and
\begin{equation}
\langle W_n^{2k}, X_0\rangle
=c_0 b_{n,0}^{2k} c_0^*,
\qquad
\langle W_n^{2k+1}, X_{-(n+1)}\rangle
=c_0 b_{n,j}^{2k+1} \tilde{c}_0^*.
\label{eq:expan222}
\end{equation}
\end{proposition}

\begin{proof}
Note that, from the definition of $W^k_n$,
\[
W_n^{2k+1}=-P_{[-n,\infty)}^\perp W_n^{2k},
\qquad
W_n^{2k+2}=-P_{(-\infty,-1]}^\perp W_n^{2k+1}.
\]
We prove (\ref{eq:expan111}) by induction.
First, from (\ref{eq:causal111}) and (\ref{eq:MA428}), we have
\[
W_n^0=P_{(-\infty,-1]}^{\perp} X_0
=c_0 \xi_0
=c_0
\sum_{j=0}^\infty b_{n,j}^0 \xi_j.
\]
For $k=0,1,\dots$, assume that
$W_n^{2k}=c_0\sum_{j=0}^\infty b_{n,j}^{2k} \xi_j$.
Then, by (\ref{eq:proj555}),
\[
\begin{aligned}
W_n^{2k+1}
&=-P_{[-n,\infty)}^{\perp}
\left(c_0\sum_{j=0}^{\infty} b_{n,j}^{2k} \xi_j\right)
=c_0\sum_{j=0}^{\infty}
\left(\sum_{l=0}^\infty b_{n,l}^{2k} \beta_{n+j+l+1}
\right)\tilde{\xi}_{n+j+1}\\
&=c_0\sum_{j=0}^\infty b_{n,j}^{2k+1} \tilde{\xi}_{n+j+1},
\end{aligned}
\]
and, by (\ref{eq:proj666}),
\[
\begin{aligned}
W_n^{2k+2}
&=-P_{(-\infty,-1]}^{\perp}
\left(c_0\sum_{j=0}^\infty b_{n,j}^{2k+1} \tilde{\xi}_{n+j+1}\right)\\
&=c_0\sum_{j=0}^{\infty}
\left(\sum_{l=0}^{\infty} b_{n,l}^{2k+1} \beta_{n+j+l+1}^*
\right)\xi_j
=c_0\sum_{j=0}^{\infty}
b_{n,j}^{2k+2} \xi_j.
\end{aligned}
\]
Thus (\ref{eq:expan111}) follows.
We obtain the first (resp., second) equality in (\ref{eq:expan222}) from
the first (resp., second) equalities in (\ref{eq:expan111}) and (\ref{eq:MA428}).
\end{proof}

Lemma \ref{lem:beta-proper648} also allows one to define,
for $n\in\mathbb{N}$ and $k\in\mathbb{N}\cup\{0\}$,
the sequences $\{\tilde{b}_{n,j}^k\}_{j=0}^{\infty}\in \ell_{2+}^{q\times q}$ by
the recursion
\begin{equation}
\tilde{b}_{n,j}^0=\delta_{0j}I_q,
\ \ \
\tilde{b}_{n,j}^{2k+1}
=\sum_{l=0}^\infty
\tilde{b}_{n,l}^{2k} \beta_{n+j+l+1}^*,
\ \ \
\tilde{b}_{n,j}^{2k+2}
=\sum_{l=0}^{\infty}
\tilde{b}_{n,l}^{2k+1} \beta_{n+j+l+1}.
\label{eq:recursion2}
\end{equation}

\begin{proposition}\label{prop:positiv528}
We assume {\rm (A)}. Then,
for $n\in\mathbb{N}$ and $k\in\mathbb{N}\cup\{0\}$, we have
$b_{n,0}^{2k} \ge 0$ and $\tilde{b}_{n,0}^{2k} \ge 0$.
\end{proposition}

\begin{proof}
Let $A=M_{[-n,\infty)}^X$ and $B=M_{(-\infty,-1]}^X$. Then,
in the same way as the proof of (\ref{eq:PL444}) in
Theorem \ref{thm:proj02}, we have
\[
\langle W_n^{2k}, X_0\rangle
=
\begin{cases}
\langle P_B^{\perp}(P_A^{\perp}P_B^{\perp})^m X_0, P_B^{\perp}(P_A^{\perp}P_B^{\perp})^m X_0\rangle, &
\mbox{$k=2m:$ even},\\
\langle (P_A^{\perp}P_B^{\perp})^{m+1}X_0, (P_A^{\perp}P_B^{\perp})^{m+1}X_0\rangle,
& \mbox{$k=2m+1:$ odd}.
\end{cases}
\]
This and the first equality in (\ref{eq:expan222}) give
$c_0 b_{n,0}^{2k}c_0^* \ge 0$ or $b_{n,0}^{2k} \ge 0$.
The second equality follows from the first one applied to $\{\tilde{X}_k\}$.
\end{proof}


\section{Representation theorems}\label{sec:5}

In this section, we develop explicit representations for the finite predictor coefficients,
the finite prediction error covariances and the PACF of a $q$-variate stationary process
$\{X_k\}$, in terms of the sequence $\{\beta_j\}$ defined in Section \ref{sec:4}.
We focus on the one-step ahead predictions to keep the notation simple.

In deriving the representation theorems for the finite predictors of a
$q$-variate stationary process $\{X_k\}$, the following \textit{intersection of past and future}
property of $\{X_k\}$ plays a key role:
\[
M_{(-\infty,-1]}^X\cap M_{[-n,\infty)}^X=M_{[-n,-1]}^X,\qquad n=1,2,\dots.
\tag{\rm IPF}
\]
A useful sufficient condition for (IPF) is the following \textit{minimality} condition:
\[
\begin{aligned}
&\mbox{$\{X_k\}$ has spectral density $w$ satisfying $\det w(e^{i\theta})>0$, $\sigma$-a.e.},\\
&\mbox{and $w^{-1}\in L^{q\times q}_1(\T)$}.
\end{aligned}
\tag{\rm M}
\]
In fact, by \cite[Corollary 3.6]{IKP}, (M) implies (IPF). The condition (M) also implies 
(A) by \cite[Lemma 2.5 and Theorem 2.8]{M60}, or more directly by
\[
\begin{aligned}
\vert \log \det w\vert 
&= q\vert \log (\det w)^{1/q}\vert 
\le q\left\{ (\det w)^{1/q} + (\det w)^{-1/q} \right\}\\
&= q\left\{ (\lambda_1\cdots\lambda_q)^{1/q} + (\lambda_1^{-1}\cdots \lambda_q^{-1})^{1/q} \right\}\\
&\le q\left\{ \frac{\lambda_1 + \cdots\ + \lambda_q}{q} 
+ \frac{\lambda_1^{-1} + \cdots + \lambda_q^{-1}}{q} \right\}= \Tr w + \Tr w^{-1},
\end{aligned}
\]
where $\lambda_1,\dots,\lambda_q$ denote the eigenvalues of $w$ and we have used the inequality 
$\vert \log y\vert \le y + (1/y)$ for $y>0$.

The property (IPF) is closely related to the property
\[
M_{(-\infty,-1]}^X\cap M_{[0,\infty)}^X=\{0\}
\tag{\rm CND}
\]
called \textit{complete nondeterminacy} by \cite{Sar}.
In fact, by \cite[Theorem 3.5]{IKP}, (IPF) and (CND) are equivalent under (A).
The condition (CND) is also closely related to the \textit{rigidity}\/ for matrix-valued Hardy
functions (see \cite{KIP}).
It should be noticed that if $\{X_k\}$ satisfies (IPF), then so does the time-reversed
process $\{\tilde{X}_k\}$, and that the same holds for (M) and (CND).

Recall $W_n^k$ from (\ref{eq:W-even216}) and (\ref{eq:W-odd}). 
The next proposition is a direct consequence of (\ref{eq:PL333}) in Theorem \ref{thm:proj02}.

\begin{proposition}\label{prop:basic-rep123}
We assume {\rm (IPF)}. Then, for $n\in\mathbb{N}$, we have
\[
P_{[-n,-1]}^{\perp} X_0 = \sum_{k=0}^{\infty} W_n^{k},
\]
the sum converging strongly in $M^q$.
\end{proposition}

\begin{proof}
The equality follows
from (IPF) and (\ref{eq:PL333}) in Theorem \ref{thm:proj02}
applied to $A=M_{[-n,\infty)}^X$, $B=M_{(-\infty,-1]}^X$ and $a=X_0^j$, $j=1,\dots,q$.
\end{proof}

Under (A), and for $n\in\mathbb{N}$ and $k=1,\dots,n$,
the \textit{forward and backward finite predictor coefficients}\/
$\phi_{n,k}\in \C^{q\times q}$ and $\tilde{\phi}_{n,k}\in \C^{q\times q}$, respectively,
of a $q$-variate stationary process $\{X_k\}$ are defined by
\begin{align}
&P_{[-n,-1]}X_0=\phi_{n,1}X_{-1}+\cdots+\phi_{n,n}X_{-n},
\label{eq:forward-phi123}\\
&P_{[-n,-1]}X_{-(n+1)}
=\tilde\phi_{n,1}X_{-n}+\cdots+\tilde\phi_{n,n}X_{-1}.
\label{eq:backward-phi123}
\end{align}

Recall 
$c_0$, $\tilde{c}_0$, $\beta_j$, $b_{n,j}^{2k}$ and $\tilde{b}_{n,j}^{2k}$ from 
(\ref{eq:MAAR111}), (\ref{eq:MAAR222}), (\ref{eq:beta-def667}), (\ref{recursion}) and (\ref{eq:recursion2}), respectively. 
Here is the representation theorem for $\phi_{n,n}$ and $\tilde\phi_{n,n}$, which are
closely related to the PACF of $\{X_k\}$.

\begin{theorem}\label{thm:pac947}
We assume {\rm (A)} and {\rm (IPF)}. Then, for $n\in\mathbb{N}$,
\[
\phi_{n,n}
=c_0\sum_{k=0}^\infty
\left(\sum_{j=0}^\infty
b_{n,j}^{2k}\beta_{n+j}\right)\tilde{c}_0^{-1},
\qquad
\tilde{\phi}_{n,n}
=\tilde{c}_0\sum_{k=0}^\infty
\left(\sum_{j=0}^\infty
\tilde{b}_{n,j}^{2k}\beta_{n+j}^*\right)c_0^{-1}.
\]
\end{theorem}

\begin{proof}
Since $P_{[-n,-1]}^{\perp} X_0\equiv-\phi_{n,n}X_{-n}
\equiv -\phi_{n,n}\tilde{c}_0\tilde{\xi}_{n}$
mod $M_{[-n+1,\infty)}^X$, we have
\[
\langle P_{[-n,-1]}^{\perp} X_0,\tilde{\xi}_{n}\rangle
=-\phi_{n,n}\,\tilde{c}_0\langle \tilde{\xi}_{n},\tilde{\xi}_{n}\rangle
=-\phi_{n,n}\,\tilde{c}_0.
\]
On the other hand, from Propositions \ref{prop:basic-rep123} and
\ref{prop:forward638} and Lemma \ref{lem:beta-xi897}, we get
\[
\langle P_{[-n,-1]}^{\perp} X_0, \tilde{\xi}_{n}\rangle
=\sum_{k=0}^\infty \langle W_n^{2k}, \tilde{\xi}_{n}\rangle
=-c_0\sum_{k=0}^{\infty}
\left(\sum_{j=0}^\infty b_{n,j}^{2k}\beta_{n+j}\right).
\]
Thus the first formula follows. We obtain the second formula by applying
the first one to the time-reversed process $\{\tilde{X}_k\}$,
\end{proof}

For $n=0,1,\dots$, we define
the \textit{forward and backward finite prediction error covariances}\/ $v_n$ and $\tilde{v}_n$,
respectively, of a $q$-variate stationary process $\{X_k\}$ by
$v_0=\tilde{v}_0=\langle X_0, X_0\rangle$ and
\begin{align}
&v_n:=\langle P_{[-n,-1]}^{\perp} X_{0}, P_{[-n,-1]}^{\perp} X_{0}\rangle,\qquad n=1,2,\dots,
\label{eq:var459}\\
&\tilde{v}_n:=\langle P_{[-n,-1]}^{\perp} X_{-(n+1)}, P_{[-n,-1]}^{\perp} X_{-(n+1)}\rangle,
\qquad n=1,2,\dots.
\label{eq:var-tilde777}
\end{align}
Notice that $\tilde{v}_n$ (resp., $v_n$) is the forward (resp., backward)
finite prediction error covariance of the time-reversed process $\{\tilde{X}_k\}$.
In this paper, under (A), we fix the definition of the \textit{partial autocorrelation function}\/ (PACF)
$\alpha_n$ of $\{X_k\}$ by
\[
\alpha_n:=
\begin{cases}
(v_0)^{-1/2} \langle X_0, X_{-1}\rangle (\tilde{v}_0)^{-1/2}, & n=1,\\
(v_{n-1})^{-1/2}
\langle P_{[-n+1,-1]}^\perp X_0,P_{[-n+1,-1]}^\perp X_{-n}\rangle
(\tilde{v}_{n-1})^{-1/2}, & n = 2, 3, \dots
\end{cases}
\]
(cf.\ \cite{De}).

The next theorem gives explicit representations for $v_n$, $\tilde{v}_n$ and $\alpha_n$.

\begin{theorem}\label{thm:cov564}
We assume {\rm (A)} and {\rm (IPF)}. Then, for $n\in\mathbb{N}$, we have
\begin{align}
&v_n
= c_0\left( \sum_{k=0}^{\infty} b_{n,0}^{2k}\right)c_0^*, \qquad
\tilde{v}_n
= \tilde{c}_0\left( \sum_{k=0}^{\infty} \tilde{b}_{n,0}^{2k}\right)\tilde{c}_0^*,
\label{eq:cov111}\\
&\langle P_{[-n+1,-1]}^\perp X_0,P_{[-n+1,-1]}^\perp X_{-n}\rangle
= c_0\left( \sum_{k=0}^{\infty} b_{n,0}^{2k+1}\right)\tilde{c}_0^*.
\label{eq:cov333}
\end{align}
\end{theorem}

\begin{proof}
First, by (IPF) and (\ref{eq:PL444}) in Theorem \ref{thm:proj02}
applied to $A=M_{[-n,\infty)}^X$, $B=M_{(-\infty,-1]}^X$ and
$a_1=X_0^i$, $a_2=X_0^j\ (i,j=1,\dots,q)$, we have
$v_n
= \sum_{k=0}^{\infty}\langle W_n^{2k}, X_0\rangle$.
This and (\ref{eq:expan222}) give the first equality in (\ref{eq:cov111}).
Next, we obtain the second equality in (\ref{eq:cov111}) by applying
the first one to the time-reversed process $\{\tilde{X}_k\}$.
Finally, by (IPF) and (\ref{eq:PL555}) in Theorem \ref{thm:proj02}
applied to $A=M_{[-n,\infty)}^X$, $B=M_{(-\infty,-1]}^X$ and
$a=X_0^i$, $b=X_{-(n+1)}^j\ (i,j=1,\dots,q)$, we have
\[
\langle P_{[-n,-1]}^{\perp} X_{0}, P_{[-n,-1]}^{\perp} X_{-(n+1)}\rangle
= \sum_{k=0}^{\infty}\langle W_n^{2k+1}, X_{-(n+1)}\rangle.
\]
This and (\ref{eq:expan222}) give (\ref{eq:cov333}).
\end{proof}

We can prove
\[
\langle P_{[-n+1,-1]}^\perp X_0,P_{[-n+1,-1]}^\perp X_{-n}\rangle
= \phi_{n,n}\tilde{v}_{n-1},\qquad n=2,3,\dots,
\]
in the same way as in the univariate case
(cf. Corollary 5.2.1 in \cite{BD}). From this,
we have
\begin{equation}
\alpha_n = (v_{n-1})^{-1/2} \phi_{n,n} (\tilde{v}_{n-1})^{1/2},\qquad n=1,2,\dots,
\label{eq:pacf-pha428}
\end{equation}
and so Theorem \ref{thm:pac947} with (\ref{eq:cov111}) in
Theorem \ref{thm:cov564} gives another explicit representation of $\alpha_n$.

We turn to the representation of all the finite predictor coefficients
$\phi_{n,j}$ and $\tilde{\phi}_{n,j}$. It turns out that, to deal with this problem,
we need to assume the minimality (M)
which is more stringent than (IPF) or (CND).
A $q$-variate stationary process $\{X_k\}$ satisfying (M) has a
dual process $\{X'_k: k\in\Z\}$, characterized by the biorthogonality relation
$\langle X_j,X'_k\rangle=\delta_{jk}I_q$; see \cite{M60} for more information. 
Recall $a_k$ and $\tilde{a}_k$ from (\ref{eq:MAAR111}) and (\ref{eq:MAAR222}), respectively. 
The dual process $\{X'_k\}$ admits the following two MA representations:
\begin{equation}
X'_n=-\sum_{k=0}^\infty a_k^* \xi_{n+k},
\qquad
X'_{-n}=-\sum_{k=0}^\infty \tilde{a}_k^* \tilde{\xi}_{n+k},\qquad n\in\Z.
\label{eq:MAdual872}
\end{equation}
Here notice that (M) implies
\begin{equation}
\{a_k\}, \{\tilde{a}_k\}\in \ell_{2+}^{q\times q}.
\label{eq:square-integrable111}
\end{equation}
By (\ref{eq:square-integrable111}),
we can also define, for $n\in\mathbb{N}$ and $k, j\in\mathbb{N}\cup\{0\}$,
\begin{align*}
&\phi_{n,j}^{2k}
:=c_0\sum_{l=0}^\infty b_{n,l}^{2k}
a_{j+l}.
\qquad
\phi_{n,j}^{2k+1}
:=c_0\sum_{l=0}^\infty b_{n,l}^{2k+1}
\tilde a_{j+l},\\
&\tilde\phi_{n,j}^{2k}
:=\tilde{c}_0\sum_{l=0}^\infty\tilde b_{n,l}^{2k}
\tilde a_{j+l},
\qquad
\tilde\phi_{n,j}^{2k+1}
:=\tilde{c}_0\sum_{l=0}^\infty\tilde b_{n,l}^{2k+1}
a_{j+l}.
\end{align*}

Here is the representation theorem for the finite predictor coefficients.

\begin{theorem}\label{thm:fpc548}
We assume {\rm (M)}. Then, for $n=1,2,\dots$ and $j=1,\dots,n$,
\begin{equation}
\phi_{n,j}
=\sum_{k=0}^\infty\{\phi_{n,j}^{2k}+\phi_{n,n-j+1}^{2k+1}\},
\qquad
\tilde\phi_{n,j}
=\sum_{k=0}^\infty
\{\tilde\phi_{n,j}^{2k}+\tilde\phi_{n,n-j+1}^{2k+1}\}.
\label{eq:phi-rep}
\end{equation}
\end{theorem}

\begin{proof}
From $\langle X_k,X'_j\rangle=\delta_{kj}I_q$, we have
$\langle P_{[-n,-1]}^\perp X_0, X'_{-j}\rangle = -\phi_{n,j}$ for $j=1,\dots,n$,
and, from Proposition \ref{prop:basic-rep123}, we find that
\[
\langle P_{[-n,-1]}^\perp X_0, X'_{-j}\rangle
=\sum_{k=0}^{\infty}\{ \langle W_n^{2k}, X'_{-j}\rangle + \langle W_n^{2k+1}, X'_{-j}\rangle\}.
\]
Moreover, from Proposition \ref{prop:forward638} and
(\ref{eq:MAdual872}) rewritten as
\[
X'_{-j}=-\sum_{l=-j}^\infty a_{j+l}^* \xi_l,
\qquad
X'_{-j}=-\sum_{l=-(n-j+1)}^\infty
\tilde{a}_{n-j+l+1}^* \tilde{\xi}_{n+l+1},
\]
we have
\[
\begin{aligned}
&\langle W_n^{2k},X'_{-j}\rangle
=-c_0\sum_{l=0}^\infty b_{n,l}^{2k}a_{j+l}
=-\phi_{n,j}^{2k},\\
&\langle W_n^{2k+1},X'_{-j}\rangle
=-c_0\sum_{l=0}^\infty
b_{n,l}^{2k+1}\tilde a_{n-j+l+1}
=-\phi_{n,n-j+1}^{2k+1}.
\end{aligned}
\]
Combining, we obtain the first equality in (\ref{eq:phi-rep}).
Its second equality follows from the first one
applied to the time-reversed process $\{\tilde{X}_k\}$.
\end{proof}


\section{Applications to long-memory processes}\label{sec:6}

In this section, we apply the representation theorems in Section \ref{sec:5} to a
$q$-variate FARIMA process with common fractional differencing order for all components and derive the
asymptotics of the finite prediction error covariances and the PACF as well as
that of the finite predictor coefficients, and establish Baxter's inequality.


\subsection{Univariate FARIMA processes}

We start with some properties of univariate FARIMA$(0,d,0)$ processes which we need in
our perturbation technique below. This technique reduces the study of asymptotic properties of
multivariate FARIMA processes to that
of the corresponding problems for univariate FARIMA$(0,d,0)$ processes.

For $d\in (-1/2, 1/2)\setminus \{0\}$, let $\{Y_k: k\in\Z\}$ be a univariate FARIMA$(0,d,0)$ process
with spectral density
\begin{equation}
w_Y(e^{i\theta}) = |1-e^{i\theta}|^{-2d},\qquad \theta\in (-\pi,\pi)
\label{eq:spY752}
\end{equation}
(see \cite{GJ, Ho}; see also \cite[\S 13.2]{BD}).
Then the variance $u_0:=E[\vert Y_0\vert^2]$ is equal to
$\Gamma(1-2d)/\Gamma(1-d)^2$ and the
$n$-th finite predictor coefficients $\psi_{n,n}$ of $\{Y_k\}$ (see (\ref{eq:forward-phi123})) 
are given by
\begin{equation}
\psi_{n,n} = \frac{d}{n-d},\qquad n\in\N.
\label{eq:YPACF857}
\end{equation}
Let $u_n$ be the finite prediction error variance of $\{Y_k\}$ defined by
(\ref{eq:var459}) with $\{X_k\}$ replaced by $\{Y_k\}$, for which we use the
notation $u_n$ rather than $v_n$. Then the
Durbin--Levinson algorithm implies $u_n=u_0\prod_{k=1}^n\{1 - (\psi_{k,k})^2\}$,
whence
\begin{equation}
u_n = \frac{\Gamma(n+1-2d)\Gamma(n+1)}{\Gamma(n+1-d)^2},\qquad n=0,1,\dots.
\label{eq:unifarimacov857}
\end{equation}

For $u_n$, we present next its precise asymptotic behavior.

\begin{proposition}\label{prop:arima-cov659}
For $d\in (-1/2, 1/2)\setminus \{0\}$, we have
$u_n = 1 + (d^2/n) + O(n^{-2})$ as $n\to\infty$.
\end{proposition}

\begin{proof}
By Stirling's formula
$\Gamma(x) = \sqrt{2\pi}e^{-x}x^{x+(1/2)}\{1 + (1/12x) + O(x^{-2})\}$ as $x\to\infty$ and
\[
\left(\frac{n+1-2d}{n+1}\right)^{-d} = 1 + \frac{2d^2}{n} + O(n^{-2}),\ \ \
\frac{\sqrt{(n+1-2d)(n+1)}}{(n+1-d)} = 1 + O(n^{-2})
\]
as $n\to\infty$, we have, as $n\to\infty$,
\[
\begin{aligned}
u_n&=\frac{\Gamma(n+1-2d)\Gamma(n+1)}{\Gamma(n+1-d)^2} \\
&= \left( 1 - \frac{d}{n+1-d} \right)^{n+1-d}
 \left( 1 + \frac{d}{n+1-d} \right)^{n+1-d} \left\{ 1 + \frac{2d^2}{n} + O(n^{-2}) \right\}.
\end{aligned}
\]
On the other hand, by l'Hopital's rule, we have, for $a\in\mathbb{R}$,
\[
\left( 1 + \frac{a}{x}\right)^x = e^a -\frac{a^2e^a}{2x} + O(x^{-2}),\qquad x\to\infty,
\]
whence, as $n\to\infty$,
\[
\left( 1 - \frac{d}{n+1-d} \right)^{n+1-d}
 \left( 1 + \frac{d}{n+1-d} \right)^{n+1-d}
=\left\{ 1 - \frac{d^2}{n} + O(n^{-2}) \right\}.
\]
Combining, we obtain the proposition.
\end{proof}

Since
\[
1 - e^{i\theta} =
\begin{cases}
\vert 1 - e^{i\theta} \vert e^{(i/2)(\theta - \pi)} & \mbox{if $0<\theta<\pi$},\\
\vert 1 - e^{i\theta} \vert e^{(i/2)(\theta + \pi)} & \mbox{if $-\pi<\theta<0$},
\end{cases}
\]
the phase function
\[
\Omega(e^{i\theta}):=\overline{(1 - e^{i\theta})^{-d}}/( 1- e^{i\theta})^{-d}
\]
of the univariate FARIMA$(0,d,0)$ process $\{Y_k\}$ above is given by
\begin{equation}
\Omega(e^{i\theta})=
\begin{cases}
e^{id(\theta - \pi)} & \mbox{if $0<\theta<\pi$},\\
e^{id(\theta + \pi)} & \mbox{if $-\pi<\theta<0$}.
\end{cases}
\label{eq:phase-ft565}
\end{equation}
Therefore, the minus of the Fourier coefficients of the phase function $\Omega(e^{i\theta})$
for $\{Y_k\}$, which we write as $\rho_n$ rather than $\beta_n$, are given by
\begin{equation}
\rho_n
=-\int_{-\pi}^{\pi}e^{-in\theta}\Omega(e^{i\theta})\frac{d\theta}{2\pi}
=\frac{\sin(\pi d)}{\pi(n-d)},\qquad n\in\Z.
\label{eq:defrho873}
\end{equation}
One can also obtain (\ref{eq:defrho873}) using \cite[Remark 1 and Lemma 4.4]{BIK}.

\begin{lemma}\label{lem:conv326}
Let $\{s_k\}_{k=-\infty}^{\infty}$ be a complex sequence such that
$\sum_{k=-\infty}^\infty k^2|s_k|<\infty$. Then, we have
\[
\lim_{n\to\infty}n
\left(\rho_n^{-1}\sum_{k=-\infty}^{\infty} \rho_{n-k}s_k
-\sum_{k=-\infty}^{\infty} s_k\right)
=\sum_{k=-\infty}^{\infty} ks_k.
\]
\end{lemma}

\begin{proof}
Since $\rho_{n-k}/\rho_n=(n-d)/(n-k-d)$, we have
\begin{equation}
n\left(\rho_n^{-1}\sum_{k=-\infty}^{\infty} \rho_{n-k}s_k
-\sum_{k=-\infty}^{\infty} s_k\right)
=\sum_{k=-\infty}^{\infty}
\frac{nk s_k}{n-k-d}.
\label{eq:nkr447}
\end{equation}
For $k\in\Z$, the function $f_{k, d}: \Z\to [0,\infty)$ defined by
\[
f_{k, d}(n):=\left\vert\frac{nk}{n-k-d}\right\vert
=\left\vert k+\frac{k(k+d)}{n-(k+d)}\right\vert
\]
takes the maximum value at either $n=k-1, k$, or $k+1$, whence
\[
\max_{n\in\N} f_{k,d}(n)
\le \max\left\{
\frac{k(k-1)}{1+d},
\frac{k^2}{\abs{d}},
\frac{k(k+1)}{1-d}
\right\}\le ck^2
\]
for some $c\in (0,\infty)$.
Therefore, we have dominated convergence, as $n\to\infty$, on
the right of (\ref{eq:nkr447}), and the sum converges to
$\sum_{k=-\infty}^{\infty} ks_k$, as desired.
\end{proof}


\subsection{Multivariate FARIMA processes}
Let $\overline{\D}:=\{z\in\C: \vert z\vert\le 1\}$ be the closed unit disk in $\C$.
We consider the following condition for $g:\T\to \C^{q\times q}$:
\[
\begin{aligned}
&\mbox{the entries of $g(z)$ are rational functions in $z$ that have}\\
&\mbox{no poles on $\overline{\D}$, and $\det g$ has no zeros on $\overline{\D}$.}
\end{aligned}
\tag{C}
\]
The condition (C) implies that $g$ is an outer function in $H_2^{q\times q}(\T)$.

\begin{lemma}\label{lem:rational555}
For $g:\T\to \C^{q\times q}$ with {\rm (C)}\/, there exists $\tilde{g}: \T\to \C^{q\times q}$
that satisfies {\rm (C)}\/ and
\begin{equation}
g(e^{-i\theta})g(e^{-i\theta})^*=\tilde{g}(e^{i\theta})\tilde{g}(e^{i\theta})^*.
\label{eq:gtilde111}
\end{equation}
The function $\tilde{g}$ is uniquely determined from $g$ up to a constant unitary factor.
\end{lemma}

\begin{proof}
Since the entries of $g(1/z)$ are rational,
the lemma follows from the proof of Theorem 10.1 in
\cite[Chapter I]{Roz}.
\end{proof}

Let $g$ and $\tilde{g}$ be as in Lemma \ref{lem:rational555}.
As in (\ref{eq:outft628}),
we define the outer function $g_{\sharp}$ in $H_2^{q\times q}(\T)$ by
\begin{equation}
g_{\sharp}(z) := \{\tilde{g}(\overline{z})\}^*.
\label{eq:outft739}
\end{equation}
Then, $g_{\sharp}$ satisfies both (C) and
\begin{equation}
g(e^{i\theta})g(e^{i\theta})^*=g_{\sharp}(e^{i\theta})^*g_{\sharp}(e^{i\theta}),
\qquad \theta\in [-\pi,\pi).
\label{eq:rational333}
\end{equation}
It should be noticed that the proof of Theorem 10.1 in \cite[Chapter I]{Roz} is constructive,
whence so is the above proof of the existence of $\tilde{g}$ and $g_{\sharp}$.

\begin{exmp}
For $c\in\D$, let
\[
g(z)=
\left(
\begin{matrix}
1 & 0\\
1/(1-cz) & 1
\end{matrix}
\right).
\]
Then $g$ satisfies (C). From the proof of Lemma \ref{lem:rational555}
and (\ref{eq:outft739}), we obtain
\[
g_{\sharp}(z)=\frac{1}{\sqrt{1 - \vert c\vert^2 + \vert c\vert^4}}
\left(
\begin{matrix}
1 - \vert c\vert^2 & 1\\
-1 + \dfrac{1 - \vert c\vert^2}{1 - cz} & -\vert c\vert^2 + \dfrac{1}{1-cz}
\end{matrix}
\right).
\]
One can also directly check that $g_{\sharp}$ satisfies both (C) and (\ref{eq:rational333}).
\end{exmp}

Let $d\in(-1/2,1/2)\setminus\{0\}$, and let $\{X_k\}$ be
a $q$-variate stationary process which has spectral density $w$ of the form
\[
\mbox{$w(e^{i\theta})=|1-e^{i\theta}|^{-2d}g(e^{i\theta})g(e^{i\theta})^*$,
where $g:\T\to \C^{q\times q}$ satisfies (C).}
\tag{F}
\]
We call the process $\{X_k\}$ a
\textit{$q$-variate FARIMA process}.
We easily find that $\{X_k\}$ satisfies (M), whence (A) and (IPF) (see Section \ref{sec:5}).
Let $\tilde{g}$ and $g_{\sharp}$ be as in Lemma \ref{lem:rational555} and
(\ref{eq:outft739}), respectively. In what follows,
as the outer functions $h$ and $\tilde{h}$ for $\{X_k\}$ in Section \ref{sec:2}, we take
\begin{equation}
h(z)=(1-z)^{-d}g(z),\qquad \tilde{h}=(1-z)^{-d}\tilde{g}(z).
\label{eq:ratouter372}
\end{equation}
Then, $h_{\sharp}$ defined by (\ref{eq:outft628}) is given by
\begin{equation}
h_{\sharp}(z)=(1-z)^{-d}g_{\sharp}(z).
\label{eq:hsharp726}
\end{equation}
From the second equality in (\ref{eq:ratouter372}), we see that the time-reversed
process $\{\tilde{X}_k\}$ of $\{X_k\}$ is also a $q$-variate FARIMA process
satifying (F) with the same differencing order $d$ and $\tilde{g}$ as
$g$.

Let $\{c_n\}$ and $\{\tilde{c}_n\}$ be the forward and backward MA coefficients of $\{X_k\}$,
respectively (see (\ref{eq:MAAR111}) and (\ref{eq:MAAR222})). Then
\begin{equation}
c_0=h(0)=g(0),\qquad
\tilde{c}_0=\tilde{h}(0)=h_{\sharp}(0)^*=g_{\sharp}(0)^*=\tilde{g}(0).
\label{eq:c-638}
\end{equation}
The sequence $\{\beta_n\}$ for $\{X_k\}$, which is defined by (\ref{eq:beta-def667}), is given by
\[
\beta_n=-
\int_{-\pi}^{\pi}e^{-in\theta} \Omega(e^{i\theta}) g(e^{i\theta})^* g_{\sharp}(e^{i\theta})^{-1}
\frac{d\theta}{2\pi},\qquad n\in\Z,
\]
with $\Omega(e^{i\theta})$ in (\ref{eq:phase-ft565}).

We define a $q\times q$ unitary matrix $U$ by
\begin{equation}
U:=g(1)^*g_{\sharp}(1)^{-1}.
\label{eq:defU327}
\end{equation}
Recall the spectral norm $\Vert a\Vert$ of $a\in \C^{q\times q}$ from Section \ref{sec:2}.
The next proposition may be viewed as an improvement of
Proposition 4.5 in \cite{BIK}.

\begin{proposition}\label{prop:delta379}
For $d\in(-1/2,1/2)\setminus\{0\}$, let $\{X_k\}$ be a $q$-variate FARIMA process
with {\rm (F)}.
For $n\in\N$, define $\Delta_n, \Delta_n'\in \C^{q\times q}$ by
\[
\beta_n=\rho_n(I_q+\Delta_n) U = \rho_n U (I_q+\Delta_n'),
\]
respectively. Then there exists a positive constant $M$ satisfying the two conditions
\begin{align}
&\nor{\Delta_n}\le Mn^{-1},
\qquad n\in\N,
\label{eq:estimate526}\\
&\nor{\Delta_n'}\le Mn^{-1},
\qquad n\in\N.
\label{eq:estimate129}
\end{align}
\end{proposition}

\begin{proof}
We put $G(z):=\{g(1/\overline{z})\}^*$. Then
$g(e^{i\theta})^*g_{\sharp}(e^{i\theta})^{-1}=G(e^{i\theta})g_{\sharp}(e^{i\theta})^{-1}$ holds.
By the property (C) for $g$ and $g_{\sharp}$, there exists an open annulus $A$ containing
the unit circle
$\T$ such that both $G(z)$ and $g_{\sharp}(z)^{-1}$ are holomorphic in $A$, whence
$G(z)g_{\sharp}(z)^{-1}$ has the Laurent series expansion
\[
G(z)g_{\sharp}(z)^{-1}=\sum_{k=-\infty}^\infty s_kz^k,\qquad z\in A.
\]
Since $A\supset \T$,
the entries of $s_k$ decay exponentially as $k\to\pm\infty$. Moreover, since
$\beta_n=\sum_{k=-\infty}^\infty \rho_{n-k}s_k$ and $U=\sum_{k=-\infty}^{\infty}s_k$,
we have
\begin{align*}
&\Delta_n
= \left(\rho_n^{-1}\sum_{k=-\infty}^\infty \rho_{n-k}s_k - \sum_{k=-\infty}^{\infty}s_k\right)
U^{-1},\\
&\Delta_n'
= U^{-1}
\left(\rho_n^{-1}\sum_{k=-\infty}^\infty \rho_{n-k}s_k - \sum_{k=-\infty}^{\infty}s_k\right).
\end{align*}
Therefore, the proposition follows from Lemma \ref{lem:conv326}.
\end{proof}


\subsection{Asymptotics of the finite prediction error covariances}\label{subsec:6.3}

In this section, we derive the precise asymptotics of
the finite prediction error covariance matrices for $q$-variate FARIMA processes with (F).

For $d\in(-1/2,1/2)\setminus\{0\}$, let $\{X_k\}$ be a $q$-variate
FARIMA process with (F). Let $v_n$ and $\tilde{v}_n$ be the forward and backward finite prediction
error covariances of $\{X_k\}$ defined by (\ref{eq:var459}) and (\ref{eq:var-tilde777}), respectively.
We define the \textit{forward and backward infinite prediction
error covariances\/} $v_{\infty}\in \C^{q\times q}$ and $\tilde{v}_{\infty}\in \C^{q\times q}$,
respectively, of $\{X_k\}$ by
\begin{align}
&v_{\infty}:=\langle P_{(-\infty,-1]}^{\perp} X_{0}, P_{(-\infty,-1]}^{\perp} X_{0}\rangle
= c_0c_0^*,
\label{eq:v-inft345}\\
&\tilde{v}_{\infty}:=\langle P_{[1,\infty)}^{\perp} X_{0}, P_{[1,\infty)}^{\perp} X_{0}\rangle
= \tilde{c}_0\tilde{c}_0^*,
\label{eq:v-tild-inft345}
\end{align}
where $\{c_n\}$ and $\{\tilde{c}_n\}$ are the forward and backward MA coefficients of $\{X_k\}$,
respectively (see (\ref{eq:MAAR111}) and (\ref{eq:MAAR222})).
It should be noticed that $\tilde{v}_{\infty}$ (resp., $v_{\infty}$) is
the forward (resp., backward) infinite prediction
error covariance of the time-reversed process $\{\tilde{X}_k\}$.

\begin{theorem}\label{thm:asymp752}
For $d\in(-1/2,1/2)\setminus\{0\}$, let $\{X_k\}$ be a $q$-variate FARIMA process with {\rm (F)}.
Then (\ref{eq:v-asymp842}) and (\ref{eq:v-tilde-asymp842}) hold.
\end{theorem}

\begin{proof}
Let $u_n$ be as in (\ref{eq:unifarimacov857}); it is
the $n$-th finite prediction error variance for a univariate fractional
ARIMA$(0,d,0)$ process $\{Y_k\}$ with spectral density (\ref{eq:spY752}).
We prove
the assertion (\ref{eq:v-asymp842}) by comparing
$v_{n}$ with $u_n$.

From the representation of $v_{n}$ in (\ref{eq:cov111}), we have
\[
v_{n} - v_{\infty}= c_0\left(\sum_{k=1}^{\infty} b_{n,0}^{2k}\right)c_0^*.
\]
Similarly, $u_{n}$ can be
expressed, in terms of $\{\rho_j\}$ in (\ref{eq:defrho873}) only, as
\[
u_{n} - 1 = \sum_{k=1}^{\infty} r_{n,0}^{2k},
\]
where, for $n\in\N$ and $k\in\N\cup\{0\}$, $\{r_{n,j}^k\}_{k=0}^{\infty}\in \ell_{2+}$ is
the analogue of
$\{b_{n,l}^k\}_{k=0}^{\infty}$ for $\{Y_k\}$, defined by the recursion
\begin{equation}
r_{n,j}^0=\delta_{0j},
\qquad
r_{n,j}^{k+1}
=\sum_{l=0}^\infty
r_{n,l}^{k}\,\rho_{n+j+l+1}.
\label{eq:recursion542}
\end{equation}
Let $\Delta_n$ and $M$ be as in Proposition \ref{prop:delta379}.
Recall $U$ from (\ref{eq:defU327}).

From the definitions, we have
\[
b_{n,0}^{2}
=\sum_{l=0}^\infty
\beta_{n+l+1}\beta_{n+l+1}^\ast,\qquad
r_{n,0}^2
=\sum_{l=0}^\infty
\rho_{n+l+1}\rho_{n+l+1}.
\]
Since $U$ is unitary, we have, for $j,k\ge n$,
\[
\beta_j
\beta_k^*\\
=\rho_j\rho_k(I_q+\Delta_j)(I_q + \Delta_k^*).
\]
By Proposition \ref{prop:delta379} and the equality $(1+x)^2-1\le 2x(1+x)^2$ for $x\ge 0$, we
have
\[
\begin{aligned}
&\nor{(I_q+\Delta_j)(I_q+ \Delta_k^\ast)-I_q}
=\Vert\Delta_j + \Delta_k^\ast + \Delta_j\Delta_k^*\Vert\\
&\quad\le(1+\nor{\Delta_j})(1+\nor{\Delta_k})-1\le (1+Mn^{-1})^2-1\\
&\quad\le 2Mn^{-1}(1+Mn^{-1})^2
\end{aligned}
\]
for $j,k\ge n$.
Thus,
\[
\nor{b_{n,0}^{2} - r_{n,0}^2I_q}\le 2Mn^{-1}(1+Mn^{-1})^2 r_{n,0}^2,\qquad n\in\N.
\]
In the same way, we have, for $k=1,\dots$,
\[
\nor{b_{n,0}^{2k} - r_{n,0}^{2k}I_q}\le 2kMn^{-1}(1+Mn^{-1})^{2k} r_{n,0}^{2k},\qquad n\in\N.
\]

Take $t>1$ such that $t^2\sin(\pi\vert d\vert)<1$. Define $\tau_{2k}\in (0,\infty)$ by
\begin{equation}
(\pi^{-1}\arcsin x)^2=\sum_{k=1}^{\infty} \tau_{2k}x^{2k},\qquad \vert x\vert<1
\label{eq:tau-even251}
\end{equation}
(cf.\ Lemma 3.1 in \cite{IK2}).
Then, as in the proof of Proposition 3.2 in \cite{IK2}, there exists an $N\in\N$ such that
\[
1 + Mn^{-1}\le t,\quad r_{n,0}^{2k}\le n^{-1}\{t\sin(\pi\vert d\vert)\}^{2k}\tau_{2k}
\quad (k\in\N,\ n\ge N).
\]
Combining, we have, for $n\ge N$,
\[
\begin{aligned}
&\left\Vert n(v_{n} - v_{\infty}) - n(u_n - 1)v_{\infty}\right\Vert
\le \Vert c_0\Vert^2 \sum_{k=1}^{\infty}
n\Vert  b_{n,0}^{2k}  - r_{n,0}^{2k}I_q\Vert\\
&\qquad\le n^{-1}M\Vert c_0\Vert^2\sum_{k=1}^{\infty} 2k\tau_{2k}\{t^2\sin(\pi\vert d\vert)\}^{2k},
\end{aligned}
\]
whence $\Vert n(v_{n} - v_{\infty}) - n(u_n - 1)v_{\infty} \Vert=O(n^{-1})$ as $n\to\infty$.
This and Proposition \ref{prop:arima-cov659} yield (\ref{eq:v-asymp842}).
We obtain (\ref{eq:v-tilde-asymp842}) by applying (\ref{eq:v-asymp842}) to
the time-reversed process $\{\tilde{X}_k\}$.
\end{proof}


\subsection{Asymptotics of the PACF}

In this section, we derive the precise asymptotics of the PACF for a $q$-variate
FARIMA process $\{X_k\}$ with (F). Recall $U$ from (\ref{eq:defU327}).
As above, $\{c_n\}$ and $\{\tilde{c}_n\}$ denote the forward and backward MA coefficients of $\{X_k\}$,
respectively (see (\ref{eq:MAAR111}) and (\ref{eq:MAAR222})).

First, we consider the asymptotics of $\phi_{n,n}$ in (\ref{eq:forward-phi123}).

\begin{theorem}\label{thm:phi-asympt985}
Let $d\in(-1/2,1/2)\setminus\{0\}$, and let $\{X_k\}$ be a $q$-variate FARIMA process with {\rm (F)}.
Then
\[
\phi_{n,n}=\frac{d}{n}c_0 U\tilde{c}_0^{-1}+O(n^{-2}),
\qquad
n\to\infty.
\]
\end{theorem}

\begin{proof}
The proof is similar to that of Theorem \ref{thm:asymp752}.
From the representation of $\phi_{n,n}$ in Theorem \ref{thm:pac947}, we have
\[
\phi_{n,n}=c_0\left(\sum_{k=0}^\infty\phi_n^k\right)\tilde{c}_0^{-1}
\quad\mbox{with}\quad
\phi_n^k:=\sum_{j=0}^\infty b_{n,j}^{2k}\beta_{n+j}.
\]
Similarly, the scalar coefficient $\psi_{n,n}$ for a univariate FARIMA$(0,d,0)$ process
$\{Y_k\}$, which is given by (\ref{eq:YPACF857}), can be
expressed, in terms of $\{\rho_j\}$ in (\ref{eq:defrho873}) only, as
\[
\psi_{n, n}=\sum_{k=0}^\infty\psi_n^k
\quad\mbox{with}\quad
\psi_n^k:=\sum_{j=0}^\infty r_{n,j}^{2k}\rho_{n+j},
\]
where $r_{n,j}^k$ are defined by the recursion (\ref{eq:recursion542}).
We define $\epsilon:=d/\vert d\vert$ so that $\vert \rho_n\vert = \epsilon \rho_n$.
Let $\Delta_n$ and $M$ be as in Proposition \ref{prop:delta379}.

First, since
\[
\phi_n^0
=\beta_{n}=\rho_n (I_q+\Delta_n) U,
\qquad
\psi_n^0=\rho_n,
\]
it follows from Proposition \ref{prop:delta379} that
\[
\nor{\phi_n^0 - \psi_n^0U}\le Mn^{-1} \epsilon  \rho_n
=Mn^{-1}\epsilon\psi_n^0.
\]
Next, we have
\begin{align*}
&\phi_n^1
=\sum_{j=0}^\infty
\left(\sum_{l=0}^\infty
\beta_{n+l+1}\beta_{n+j+l+1}^\ast
\right)\beta_{n+j}, \\
&\psi_n^1
=\sum_{j=0}^\infty
\left(\sum_{l=0}^\infty
\rho_{n+l+1}\rho_{n+j+l+1}
\right)\rho_{n+j}.
\end{align*}
Then, since $U$ is unitary, we have, for $j,k,l\ge n$,
\[
\beta_j
\beta_k^\ast\beta_l\\
=\rho_j\rho_k\rho_l
(I_q+\Delta_j)
(I_q+\Delta_k^\ast)
(I_q+\Delta_l)U.
\]
By Proposition \ref{prop:delta379} and the equality $(1+x)^3-1\le 3x(1+x)^3$ for $x\ge 0$,
we have
\[
\begin{aligned}
&\nor{(I_q+\Delta_j)(I_q+\Delta_k^\ast)(I_q+\Delta_l)-I_q}\\
&\quad=\Vert\Delta_j + \Delta_k^\ast + \Delta_l + \Delta_j\Delta_k^\ast + \Delta_j\Delta_l
+ \Delta_k^\ast\Delta_l + \Delta_j\Delta_k^\ast\Delta_l\Vert\\
&\quad\le(1+\nor{\Delta_j})(1+\nor{\Delta_k})
(1+\nor{\Delta_l})-1\le (1+Mn^{-1})^3-1\\
&\quad\le 3Mn^{-1}(1+Mn^{-1})^3
\end{aligned}
\]
for $j,k,l\ge n$.
Thus,
\[
\nor{\phi_n^1 - \psi_n^1 U}\le 3Mn^{-1}(1+Mn^{-1})^3\epsilon \psi_n^1,\qquad n\in\N.
\]
In the same way, we have, for $k=0,1,\dots$,
\[
\nor{\phi_n^k-\psi_n^k U}\le (2k+1)Mn^{-1}(1+Mn^{-1})^{2k+1}\epsilon\psi_n^k,
\qquad n\in\N.
\]

Take $t>1$ such that $t^2\sin(\pi\vert d\vert)<1$. Define $\tau_{2k+1}\in (0,\infty)$ by
\begin{equation}
\pi^{-1}\arcsin x=\sum_{k=0}^{\infty} \tau_{2k+1}x^{2k+1},\qquad \vert x\vert<1
\label{eq:tau-odd251}
\end{equation}
(cf.\ Lemma 3.1 in \cite{IK2}).
Then, as in the proof of Proposition 3.2 in \cite{IK2}, there exists an $N\in\N$ such that
\[
1 + Mn^{-1}\le t,\quad \epsilon\psi_n^k\le n^{-1}\{t\sin(\pi\vert d\vert)\}^{2k+1}\tau_{2k+1}
\quad (k\in\N\cup\{0\},\ n\ge N).
\]
Combining, we have, for $n\ge N$,
\[
\begin{aligned}
&\left\Vert n\phi_{n,n} - \frac{n}{n-d}dc_0 U\tilde{c}_0^{-1}\right\Vert\\
&\quad=n\Vert \phi_{n,n} - \psi_{n,n}c_0 U\tilde{c}_0^{-1}\Vert
\le \Vert c_0\Vert \Vert \tilde{c}_0^{-1}\Vert
\sum_{k=0}^{\infty} n\left\Vert \phi_n^k - \psi_n^k U\right\Vert\\
&\quad\le n^{-1}\Vert c_0\Vert
\Vert \tilde{c}_0^{-1}\Vert M\sum_{k=0}^{\infty} (2k+1)\tau_{2k+1}\{t^2\sin(\pi\vert d\vert)\}^{2k+1},
\end{aligned}
\]
whence $\Vert n\phi_{n,n} - dc_0 U\tilde{c}_0^{-1}\Vert=O(n^{-1})$ as $n\to\infty$.
Thus the theorem follows.
\end{proof}

Recall $v_{\infty}$ and $\tilde{v}_{\infty}$ from
(\ref{eq:v-inft345}) and (\ref{eq:v-tild-inft345}), respectively.
Notice that $v_{\infty}^{-1/2}c_0$ (resp., $\tilde{v}_{\infty}^{-1/2}\tilde{c}_0$) is the
polar part of $c_0$ (resp., $\tilde{c}_0$).
Recall the PACF $\alpha_n$ of $\{X_k\}$ from Section \ref{sec:5}.
The above theorem gives the following rate of convergence for $\alpha_n$ as $n\to\infty$.

\begin{theorem}\label{thm:PACF431}
Let $d\in(-1/2,1/2)\setminus\{0\}$, and let $\{X_k\}$ be a $q$-variate FARIMA process with {\rm (F)}.
Then (\ref{eq:PACF741}) holds with the unitary matrix $V\in\C^{q\times q}$ given by
\[
V:=v_{\infty}^{-1/2}c_0\cdot U\cdot (\tilde{v}_{\infty}^{-1/2}\tilde{c}_0)^*.
\]
\end{theorem}

\begin{proof}
From the first equality in (\ref{eq:cov111}) in Theorem \ref{thm:cov564} and
Proposition \ref{prop:positiv528}, we
have $v_n\ge v_{\infty}$. Therefore, we see
from Theorem \ref{thm:asymp752} and \cite[Theorem X.3.7]{Bha} that
$\Vert v_n^{1/2} - v_{\infty}^{1/2}\Vert = O(n^{-1})$ as $n\to\infty$.
Similarly, we have $\tilde{v}_n^{1/2} = \tilde{v}_{\infty}^{1/2} + O(n^{-1})$ as $n\to\infty$.

From $v_n\ge v_{\infty}$ and \cite[Propositions V.1.6 and V.1.8]{Bha}, we have
$v_n^{-1/2}\le v_{\infty}^{-1/2}$,
so that $\Vert v_n^{-1/2}\Vert \le \Vert v_{\infty}^{-1/2}\Vert$. Hence, as $n\to\infty$,
\[
\begin{aligned}
\Vert v_n^{-1/2} - v_{\infty}^{-1/2}\Vert
&=\Vert v_n^{-1/2}(v_{\infty}^{1/2} - v_n^{1/2}) v_{\infty}^{-1/2}\Vert\\
&\le \Vert v_{\infty}^{-1/2}\Vert^2 \Vert v_n^{1/2} - v_{\infty}^{1/2}\Vert
= O(n^{-1}).
\end{aligned}
\]

Combining these with (\ref{eq:pacf-pha428}) and Theorem \ref{thm:phi-asympt985}, we have
\[
\begin{aligned}
n\alpha_n
&= v_{n-1}^{-1/2} \cdot n\phi_{n,n} \cdot \tilde{v}_{n-1}^{1/2}\\
&=\{v_{\infty}^{-1/2} + O(n^{-1})\}
\{dc_0 U\tilde{c}_0^{-1}+O(n^{-1})\}
\{\tilde{v}_{\infty}^{1/2} + O(n^{-1})\}\\
&=dv_{\infty}^{-1/2}c_0\cdot U\cdot (\tilde{v}_{\infty}^{-1/2}\tilde{c}_0)^* + O(n^{-1})
\end{aligned}
\]
as $n\to\infty$.
Thus the theorem follows.
\end{proof}

\begin{remark}
If we choose $g$ and $\tilde{g}$ so that both $g(0)\ge 0$ and $\tilde{g}(0)\ge 0$ hold, then we see from
(\ref{eq:c-638}), (\ref{eq:v-inft345}) and (\ref{eq:v-tild-inft345}) that
$c_0=v_{\infty}^{1/2}$ and $\tilde{c}_0=\tilde{v}_{\infty}^{1/2}$, whence $V=U$.
\end{remark}

\subsection{Baxter's inequality}

In this section, we present Baxter's inequality for multivariate FARIMA processes with
$0<d<1/2$.
It extends the corresponding univariate result in \cite{IK2}.

For $d\in (-1/2, 1/2)\setminus \{0\}$,
let $\{X_k\}$ be a $q$-variate FARIMA process with (F).
Recall the forward and backward AR coefficients $a_n$ and $\tilde{a}_n$ of $\{X_k\}$ from 
(\ref{eq:MAAR111}) and (\ref{eq:MAAR222}), respectively. They satisfy
\begin{align}
\left\Vert n^{1+d}a_n + \frac{1}{\Gamma(-d)}g(1)^{-1} \right\Vert &= O(n^{-1}),\qquad n\to\infty,
\label{eq:ar126}\\
\left\Vert n^{1+d}\tilde{a}_n + \frac{1}{\Gamma(-d)}\{g_{\sharp}(1)^*\}^{-1} \right\Vert
&= O(n^{-1}),
\qquad n\to\infty
\label{eq:aar463}
\end{align}
(cf.\ \cite[Lemma 2.2]{I2}). In particular, we have
\begin{align}
\lim_{n\to\infty} n^{1+d}\Vert a_n\Vert &= \frac{\Vert g(1)^{-1}\Vert}{\vert \Gamma(-d)\vert},
\label{eq:asympAR111}\\
\lim_{n\to\infty} n^{1+d}\Vert \tilde{a}_n\Vert &= \frac{\Vert \{g_{\sharp}(1)^*\}^{-1}\Vert}{\vert \Gamma(-d)\vert}.
\label{eq:asympAR222}
\end{align}
We see from (\ref{eq:asympAR111}) and (\ref{eq:asympAR222}) that
(\ref{eq:integrable432}) holds if $0<d<1/2$.

Recall $\phi_{n,k}$ and $\phi_k$ from (\ref{eq:forward-phi123}) and (\ref{eq:ipc574}), respectively.

\begin{theorem}\label{thm:Baxter01}
For $d\in (0,1/2)$, let $\{X_k\}$ be a $q$-variate FARIMA process with {\rm (F)}.
Then the forward finite and infinite predictor coefficients $\phi_{n,k}$ and
$\phi_k$, respectively, of $\{X_k\}$ satisfy
\[
\sum_{j=1}^{n}\nor{\phi_{n,j}-\phi_j}=O(n^{-d}),
\qquad
 n\to\infty.
\]
\end{theorem}

\begin{proof}
For $k=0,1,\dots$, we show by induction on $k$ that
\begin{equation}
\nor{b_{n,l}^k}
\le(1+Mn^{-1})^kr_{n,l}^k,\qquad n\in\N,\ l\in\N\cup\{0\},
\label{eq:estimate314}
\end{equation}
where $M$ is a positive constant satisfying (\ref{eq:estimate526}) and
$r_{n,l}^k$ are defined by (\ref{eq:recursion542}).
Indeed, the case $k=0$ is evident by the definitions
$b_{n,l}^0=\delta_{0l}I_q$ and $r_{n,l}^0=\delta_{0l}$.
Assuming (\ref{eq:estimate314}) for $k\ge 0$, we see from
Proposition \ref{prop:delta379} that
\[
\begin{aligned}
\nor{b_{n,l}^{k+1}}
&\le \sum_{m=0}^{\infty}\nor{b_{n,m}^k}\nor{\beta_{n+l+m+1}}\\
&\le (1+Mn^{-1})^{k+1}\sum_{m=0}^{\infty} r_{n,m}^k\rho_{n+l+m+1}
=(1+Mn^{-1})^{k+1}r_{n,l}^{k+1}.
\end{aligned}
\]
Thus (\ref{eq:estimate314}) also holds for $k+1$.

Define $\tau_{k}\in (0,\infty)$ by (\ref{eq:tau-even251}) and (\ref{eq:tau-odd251}).
Then we see from Proposition 3.2 in \cite{IK2} that,
for any $t>1$, there exits an $N\in\N$ such that
\[
nr_{n,l}^k\le \tau_k\{t\sin(\pi d)\}^k,
\quad
1+Mn^{-1}\le t\qquad (l\in\N\cup\{0\},\ k\in\N,\ n\ge N).
\]
Here we take $t>1$ such that $t^2\sin(\pi d)<1$.
Then, from (\ref{eq:estimate314}),
\begin{equation}
n\sum_{k=1}^\infty \nor{b_{n,l}^k}
\le
\sum_{k=1}^\infty \tau_k\{t^2\sin(\pi d)\}^k<\infty,\qquad l\in\N\cup\{0\},\ n\ge N.
\label{eq:estim812}
\end{equation}
From $\phi_j=c_0 a_j=c_0\sum_{l=0}^{\infty} b_{n,l}^0 a_{j+l}$ and Theorem \ref{thm:fpc548},
we have
\[
\phi_{n,j}-\phi_j
=c_0 \sum_{k=1}^\infty\sum_{l=0}^\infty
b_{n,l}^{2k}a_{j+l}
+c_0 \sum_{k=0}^\infty\sum_{l=0}^\infty
b_{n,l}^{2k+1}\tilde a_{n-j+l+1},
\]
whence
\[
\sum_{j=1}^{n} \Vert \phi_{n,j}-\phi_j\Vert
\le \sum_{j=1}^n \sum_{l=0}^{\infty} R_{j+l} \sum_{k=1}^{\infty} \Vert b_{n,l}^{k}\Vert,
\]
where $R_j=\max\{\nor{\phi_j},\nor{\tilde\phi_j}\}$.
Since $n^{1+d}R_n$ is bounded by (\ref{eq:asympAR111}) and (\ref{eq:asympAR222}),
we have, for $n\in\N$,
\[
n^{-1+d}\sum_{j=1}^n \sum_{l=j}^\infty R_{l}
\le \left\{\sup_{l\in\N} l^{1+d}R_l\right\}
\left\{\sup_{m\in\N} m^{-1+d}\sum_{j=1}^m \sum_{l=j}^{\infty} l^{-1-d}\right\}<\infty.
\]
Hence we see from (\ref{eq:estim812}) that,
for $n\ge N$,
\[
n^d\sum_{j=1}^{n}\nor{\phi_{n,j}-\phi_j}
\le \left\{\sum_{k=1}^\infty \tau_k\{r^2\sin(\pi d)\}^k\right\}
\left\{
\sup_{m\in\N} m^{-1+d}\sum_{j=1}^m \sum_{l=j}^\infty R_{l}\right\}<\infty.
\]
The desired result follows from this.
\end{proof}

Since $\phi_n=c_0 a_n$, we see from (\ref{eq:ar126}) that
\[
\left\Vert n^{1+d}\phi_n + \frac{1}{\Gamma(-d)}c_0g(1)^{-1} \right\Vert = 
O(n^{-1}),\qquad n\to\infty.
\]
In particular, 
\[
\lim_{n\to\infty} n^{1+d}\Vert \phi_n\Vert = \frac{\Vert c_0g(1)^{-1}\Vert}{\vert \Gamma(-d)\vert}.
\]
From this and \cite[Proposition 1.5.8]{BGT}, 
we obtain the following asymptotic behavior of $\sum_{j=n+1}^{\infty}\Vert \phi_j\Vert$ 
as $n\to\infty$:
\begin{equation}
\lim_{n\to\infty}n^d\sum_{j=n+1}^{\infty}\Vert \phi_j\Vert 
= \frac{\Vert c_0g(1)^{-1}\Vert}{\Gamma(1-d)}.
\label{eq:phi247}
\end{equation}

Here is Baxter's inequality for multivariate FARIMA processes with $0<d<1/2$.

\begin{theorem}\label{thm:Baxter02}
For $d\in (0,1/2)$, let $\{X_k\}$ be a $q$-variate FARIMA process with {\rm (F)},
and let $\phi_{n, k}$ and $\phi_n$ be as in Theorem \ref{thm:Baxter01}.
Then, there exists a positive constant $K$ such that (\ref{eq:Baxter597}) holds.
\end{theorem}

\begin{proof}
In view of (\ref{eq:phi247}), Theorem \ref{thm:Baxter01} gives the desired assertion.
\end{proof}

By applying Theorem \ref{thm:Baxter02} to
the time-reversed process $\{\tilde{X}_k\}$, we immediately obtain the
following backward Baxter inequality.

\begin{corollary}\label{cor:backward-baxter519}
For $d\in (0,1/2)$, let $\{X_k\}$ be a $q$-variate FARIMA process with {\rm (F)},
and let $\tilde{\phi}_{n,k}$ and $\tilde{\phi}_k$ be the backward finite and infinite
predictor coefficients, respectively, of $\{X_k\}$.
Then, there exists a positive constant $\tilde{K}$ such that
\begin{equation}
\sum_{j=1}^{n}\nor{\tilde{\phi}_{n,j} - \tilde{\phi}_j}
\le \tilde{K}\sum_{j=n+1}^{\infty}\Vert \tilde{\phi}_j\Vert,\qquad
n\in\N.
\label{eq:Baxter608}
\end{equation}
\end{corollary}

\

\

\noindent {\bf Acknowledgements}\quad

\

\noindent We would like to thank the anonymous referees for their helpful comments. 
M.\ Pourahmadi was supported by the NFS grant DMS-1309586.

\



\begin{thebibliography}{15}


\bibitem{BK}
\textsc{Baillie, R. T.} and \textsc{Kapetanios, G.} (2013).
Estimation and inference for impulse response functions from univariate strongly persistent processes.
\textit{Econom. J.} \textbf{16} 373--399.


\bibitem{Bax}
\textsc{Baxter, G.} (1962).
An asymptotic result for the finite predictor.
\textit{Math. Scand.} \textbf{10} 137--144.

\bibitem{Ber}
\textsc{Berk, K. N.} (1974).
Consistent autoregressive spectral estimates.
\textit{Ann. Statist.} \textbf{2} 489--502.


\bibitem{Bha}\textsc{Bhyatia, R.} (1997).
\textit{Matrix Analysis.}
Springer, New York.



\bibitem{Bin}
\textsc{Bingham, N. H.} (2012)
Multivariate prediction and matrix Szeg\"o theory.
\textit{Probab. Surv.} \textbf{9} 325--339.


\bibitem{BGT}\textsc{Bingham, N. H.}, \textsc{Goldie, C. M.} and \textsc{Teugels, J. L.} 
(1989). 
\textit{Regular Variation}, 2nd ed.
Cambridge University Press, Cambridge University Press.

\bibitem{BIK}
\textsc{Bingham, N. H.}, \textsc{Inoue, A.} and \textsc{Kasahara, Y.} (2012).
An explicit representation of Verblunsky coefficients.
\textit{Statist. Probab. Lett.} \textbf{82} 403--410.


\bibitem{BD}
\textsc{Brockwell, P. J.} and \textsc{Davis, R. A.} (1991).
\textit{Time Series: Theory and Methods}, 2nd ed.
Springer, New York.

\bibitem{Bu1}
\textsc{B\"uhlmann, P.} (1995).
Moving-average representation of autoregressive approximations.
\textit{Stochastic Process. Appl.} \textbf{60} 331--342.

\bibitem{Bu2}
\textsc{B\"uhlmann, P.} (1997).
Sieve bootstrap for time series.
\textit{Bernoulli} \textbf{3} 123--148.

\bibitem{CP}
\textsc{Cheng, R.} and \textsc{Pourahmadi, M.} (1993).
Baxter's inequality and convergence of finite predictors of multivariate stochastic processes.
\textit{Probab. Theory Related Fields} \textbf{95} 115--124.


\bibitem{Chung}
\textsc{Chung, C.-F.} (2001).
Calculating and analyzing impulse responses for the vector ARFIMA model.
\textit{Econom. Lett.} \textbf{71} 17--25.



\bibitem{DPS}
\textsc{Damanik, D.}, \textsc{Pushnitski, A.} and \textsc{Simon, B.} (2008).
The analytic theory of matrix orthogonal polynomials.
\textit{Surv. Approx. Theory} \textbf{4} 1--85.



\bibitem{De}
\textsc{D\'egerine, S.} (1990).
Canonical partial autocorrelation function of a multivariate time series.
\textit{Ann. Statist.} \textbf{18} 961--971.



\bibitem{Gi}
\textsc{Ginovian, M. S.} (1999).
Asymptotic behavior of the prediction error for stationary random sequences.
\textit{J. Contemp. Math. Anal.} \textbf{34} 18--36.

\bibitem{GJ}
\textsc{Granger, C. W.} and \textsc{Joyeux, R.} (1980).
An introduction to long-memory time series models and fractional
differencing.
\textit{J. Time Series Analysis}
\textbf{1} 15--29.


\bibitem{HD}
\textsc{Hannan, E. J.} and \textsc{Deistler, M.} (1988).
\textit{The Statistical Theory of Linear Systems}. Wiley, New York.

\bibitem{HL}
\textsc{Helson, H.} and \textsc{Lowdenslager, D.} (1961).
Prediction theory and Fourier series in several variables II.
\textit{Acta Math.} \textbf{106} 175--213.


\bibitem{Ho}
\textsc{Hosking, J. R. M.} (1981).
Fractional differencing.
\textit{Biometrika}
\textbf{68} 165--176.



\bibitem{IS}
\textsc{Ibragimov, I. A.} and \textsc{Solev, V. N.} (1968).
Asymptotic behavior of the prediction error of a stationary sequence with a
spectral density of special type.
\textit{Theory Probab. Appl.}
\textbf{13} 703--707.

\bibitem{ICM}
\textsc{Ing, C.-K.}, \textsc{Chiou, H.-T.}  and \textsc{Guo, M.} (2016).
Estimation of inverse autocovariance matrices for long memory processes.
\textit{Bernoulli} 
\textbf{22} 1301--1330.

\bibitem{I1}
\textsc{Inoue, A.} (2000).
Asymptotics for the partial autocorrelation function of a stationary process.
\textit{J. Anal. Math.} \textbf{81} 65--109.


\bibitem{I2}
\textsc{Inoue, A.} (2002).
Asymptotic behavior for partial autocorrelation functions of fractional ARIMA processes.
\textit{Ann. Appl. Probab.} \textbf{12} (2002) 1471--1491.


\bibitem{I3}
\textsc{Inoue, A.} (2008).
AR and MA representation of partial autocorrelation functions, with applications.
\textit{Probab. Theory Related Fields} \textbf{140} 523--551.


\bibitem{IK1}
\textsc{Inoue, A.} and \textsc{Kasahara, Y.} (2004).
Partial autocorrelation functions of the fractional ARIMA processes with negative degree of
differencing. \textit{J. Multivariate Anal.} \textbf{89} 135--147.


\bibitem{IK2}
\textsc{Inoue, A.} and \textsc{Kasahara, Y.} (2006).
Explicit representation of finite predictor coefficients and its applications.
\textit{Ann. Statist.} \textbf{34} 973--993.


\bibitem{IKP}
\textsc{Inoue, A.}, \textsc{Kasahara, Y.} and \textsc{Pourahmadi, M.} (2016).
The intersection of past and future for multivariate stationary processes.
\textit{Proc. Amer. Math. Soc.} \textbf{144} 1779--1786.



\bibitem{KB}
\textsc{Kasahara, Y.} and \textsc{Bingham, N. H.} (2014).
Verblunsky coefficients and Nehari sequences.
\textit{Trans. Amer. Math. Soc.} \textbf{366} 1363--1378.

\bibitem{KIP}
\textsc{Kasahara, Y.}, \textsc{Inoue, A.} and \textsc{Pourahmadi, M.} (2016).
Rigidity for matrix-valued Hardy functions.
\textit{Integr. Equ. Oper. Theory} \textbf{84} 289--300.


\bibitem{KK}
\textsc{Katsnelson, V. E.} and \textsc{Kirstein, B.} (1997).
On the theory of matrix-valued functions belonging to the Smirnov class.
Topics in Interpolation Theory (Leipzig, 1994), pp. 299--350,
Oper. Theory Adv. Appl. \textbf{95}, Birkh\"auser, Basel.


\bibitem{KPP}
\textsc{Kreiss, J.-P.}, \textsc{Paparoditis, E.} and \textsc{Politis, D. N.} (2011).
On the range of validity of the autoregressive sieve bootstrap.
\textit{Ann. Statist.} \textbf{39} 2103--2130.



\bibitem{M60}
\textsc{Masani, P.} (1960).
The prediction theory of multivariate stochastic processes. III.
\textit{Acta Math.} \textbf{104} 141--162.


\bibitem{MJK}
\textsc{Meyer, M.}, \textsc{Jentsch, C.} and \textsc{Kreiss, J.-P.} (2015).
Baxter's inequality and sieve bootstrap for random fields.
Working Paper, University of Mannheim.

\bibitem{MMP}
\textsc{Meyer, M.}, \textsc{McMurry, T. L.} and \textsc{Politis, D. N.} (2015).
Baxter's inequality for triangular arrays. Preprint, Department of Mathematics,
University of California, San Diego.


\bibitem{Pe}
\textsc{Peller, V. V.} (2003).
\textit{Hankel Operators and Their Applications}.
Springer, New York.



\bibitem{PGM}
\textsc{Poskitt, D. S.}, \textsc{Grose, S. D.} and \textsc{Martin, G. M.} (2015).
Higher-order improvements of the sieve bootstrap for fractionally integrated processes. 
\textit{J. Econometrics} \textbf{188} 94--110.



\bibitem{Pou}
\textsc{Pourahmadi, M.} (2001).
\textit{Foundations of Time Series Analysis and Prediction Theory}.
Wiley, New York.



\bibitem{Roz}
\textsc{Rozanov, Y. A.} (1967).
\textit{Stationary Random Processes}.
Holden-Day, San Francisco.


\bibitem{RS}
\textsc{Rupasinghe, M.} and \textsc{Samaranayake, V. A.} (2012).
Asymptotic properties of sieve bootstrap prediction intervals for FARIMA processes.
\textit{Statist. Probab. Lett.} \textbf{82} 2108--2114.


\bibitem{Sar}
\textsc{Sarason, D.} (1978).
\textit{Function Theory on the Unit Circle}.
Notes for lectures given at a Conference at Virginia Polytechnic Institute and State University.







\end{thebibliography}
\end{document}